\documentclass[12pt]{amsart}

\usepackage{amssymb}

\newtheorem{theorem}{Theorem}[section]
\newtheorem{proposition}[theorem]{Proposition}
\newtheorem{lemma}[theorem]{Lemma}
\newtheorem{corollary}[theorem]{Corollary}
\newtheorem{definition}[theorem]{Definition}
\newtheorem{remark}[theorem]{Remark}

\numberwithin{equation}{section}

\newcommand{\lra}{\longrightarrow}
\newcommand{\mc}{\mathcal}
\newcommand{\mcm}{{\mc M}_{n,d}} 
\newcommand{\mcu}{{\mc U}}
\newcommand{\mcux}{{\mc U}_{\xi}}
\newcommand{\mcw}{{\mc W}_{n,d}}
\newcommand{\mcpw}{{\mc {PW}}}

\newcommand{\mcmx}{{\mc M}_{n,\xi}}
\newcommand{\mcmxx}{{\mc M}'_{E_0,n,\xi}}
\newcommand{\mcmxxl}{{\mc M}'_{L_0,n,\xi}}

\newcommand\mcwx{{\mathcal{W}}_{n,\xi}}
\newcommand\mcwxp{{\mathcal{PW}}_{n,\xi}}

\newcommand\mcwd{{\mathcal{W}}_{1,d}}

\newcommand\picd{{\mathcal M}_{1,d}}
\newcommand\pico{{\operatorname{Pic}}^0(C)}
\newcommand\pic{\operatorname{Pic}}
\newcommand\mco{\mathcal{O}_C}
\newcommand{\rk}{\operatorname{rk}}
\newcommand{\mcmo}{{\mathcal M}_{n_0,d_0}}
\newcommand{\mcwoo}{{\mc W}_{n_0,d_0}}

\newcommand{\End}{\operatorname{End}}
\newcommand{\Hom}{\operatorname{Hom}}
\newcommand{\Image}{\operatorname{Im}}
\newcommand{\Coker}{\operatorname{Coker}}

\newcommand{\ch}{\operatorname{ch}^0}
\usepackage{color}
\begin{document}

\baselineskip=15.5pt

\title[Generalised Picard sheaves]{Stability and deformations of generalised Picard sheaves}

\author{I. Biswas}

\address{School of Mathematics, Tata Institute of Fundamental Research,
Homi Bhaba Road, Mumbai 400005, India}

\email{indranil@math.tifr.res.in}

\author{L. Brambila-Paz}

\address{CIMAT A.C., Jalisco S/N, Mineral de Valenciana
C.P. 36023, Guanajuato, Gto, M\'{e}xico}

\email{lebp@cimat.mx}

\author{P. E. Newstead}

\address{Department of Mathematical Sciences, The University of
Liverpool, Peach Street, Liverpool, L69 7ZL, England}

\email{newstead@liverpool.ac.uk}

\subjclass[2010]{14H60, 14J60}

\date{\today}

\thanks{All authors are members of the international research group VBAC. The first author
acknowledges the support of a J. C. Bose Fellowship. The second author 
acknowledges the support of CONACYT grant 251938.}

\begin{abstract}
Let $C$ be a smooth irreducible complex projective curve of genus $g \geq  2$ and $M$ the moduli space of stable vector bundles on $C$ of rank $n$ and degree $d$ with $\gcd(n,d)=1$. 
A generalised Picard sheaf is the direct image on $M$ of the tensor product of a universal bundle on $M\times C$ by the pullback of a vector bundle $E_0$ on $C$. In this paper, we investigate the stability of generalised Picard sheaves and, in the case where these are locally free, their deformations. When $g\ge3$, $n\ge2$ (with some additional restrictions for $g=3,4$) and the rank and degree of $E_0$ are coprime, this leads to the construction of a fine moduli space for deformations of Picard bundles.
\end{abstract}

\maketitle

\section{Introduction}\label{intro}

Let $C $ be a smooth irreducible complex projective curve of genus 
$g\ge2$ and $\mcm$ the moduli space of stable bundles of rank $n$ and degree $d$ on $C$, where $\gcd(n,d)=1$. Denote by $\mcu$ a universal (or Poincar\'{e}) 
bundle over $\mcm\times C$, which will remain fixed unless otherwise stated (in general, $\mcu$ is determined up to tensoring by a line bundle lifted from $\mcm$). For any vector bundle $E_0$ on $C$, the torsion-free sheaf
\[\mcw(E_0):=p_{1*}(p_2^*(E_0)\otimes\mcu)\]
on $\mcm$ is called a \textit{generalised Picard sheaf} (for convenience, we shall say simply \textit{Picard sheaf}). A similar definition can be made when $\mcm$ is replaced by the fixed determinant moduli space $\mcmx$ for any line bundle $\xi$ of degree $d$ and $\mcu$ is replaced by a universal bundle $\mcux$ on $\mcmx\times C$ (we take the restriction of our chosen universal bundle $\mcu$ on $\mcm\times C$). This gives rise to a Picard sheaf
\[\mcwx(E_0):=p_{1*}(p_2^*(E_0)\otimes\mcux)\]
on $\mcmx$. 

Picard sheaves are closely related to Fourier-Mukai transforms. If $H^1(C,E_0\otimes E)=0$ for all $E\in\mcm$ (respectively, all $E\in\mcmx$), then $\mcw(E_0)$ (respectively, $\mcwx(E_0)$) is locally free and coincides with the Fourier-Mukai transform. In this case, the Picard sheaves may be referred to as \textit{Picard bundles}. This happens in particular when $E_0$ is stable of rank $n_0$ and degree $d_0$ and $nd_0+n_0d>n_0n(2g-2)$. Our aims in this paper are to obtain new results on the slope-stability of Picard sheaves and on the deformations of Picard bundles. 

When $n=1$ and $d\ge 2g-1$, the bundles $\mcwd({\mc O}_C)$ coincide with classical Picard 
bundles, which were introduced in projectivised form in \cite{m} in a very general setting. The Picard sheaves $\mcwd({\mc O}_C)$ were studied in \cite{s}. Picard bundles appear also in the work of Gunning \cite{g,g1} as analytic bundles associated with factors of automorphy. There is an extensive literature on the stability of Picard bundles when 
$E_0={\mc O}_C$ (see, for example, \cite{k}, \cite{el}, \cite{li}, \cite{bbgn}, 
\cite{bisbn2}, \cite{bep}); the results of these papers extend easily to the case where $E_0$ is an 
arbitrary line bundle (i.e., $n_0=1$ (see Theorem \ref{th11})). For $n_0\ge2$ and $n=1$, see \cite{br} and 
\cite{hp}. When $\gcd(n,d)\ne1$, it is possible to define a \textit{projective Picard bundle} 
$\mcwxp(E_0)$ on the Zariski open subset
\begin{equation}\label{e1}
\mcmxx\,=\, \{E\, \in\, \mcmx\,\mid\, h^1(E_0\otimes E)\,=\, 0\} \, \subset\, \mcmx
\end{equation}
(see \cite{bisbn2}); this subset is empty for $nd_0+n_0d\,<\, n_0n(g-1)$ by 
Riemann-Roch. If $E_0$ is semistable and $nd_0+n_0d\,>\,n_0n(2g-2)$, then $\mcmxx\,
=\,\mcmx$. It is of course 
possible to make similar definitions for $\mcm$. Projective Picard bundles on moduli spaces of 
symplectic and orthogonal bundles have been studied in \cite{bg1} and \cite{bg2}. Picard bundles on 
Prym varieties are discussed in \cite{bep}. A study of Picard bundles on nodal curves has been 
initiated in \cite{bp}; very recently, some of the results of \cite{bisbn2} have been generalised to nodal curves \cite{abk}.

Deformations of $\mcwd({\mc O}_C)$ were studied by Kempf \cite{k1} and Mukai \cite{mu}. When $\gcd(n,d)=1$, this was extended to $\mcwx(E_0)$ with $E_0$ semistable in \cite{bisbn1}. The main object of \cite{bisbn1} was to show that all deformations arise in a natural way from those of $E_0$.

Our first aim is to present a treatment of stability properties of Picard bundles and sheaves on $\mcm$
and $\mcmx$ with $E_0$ any stable bundle, including all known results and significant
new ones. When $\gcd(n,d)=1$, we also study deformations of the Picard bundles $\mcw(E_0)$, thus generalising the results of \cite{bisbn1} and showing that these deformations again arise in a natural way from those of $E_0$. For $n\ge2$, this leads to the construction of a smooth irreducible moduli space for Picard bundles and an identification of this moduli space with an open subset in a moduli space ${\mc M}_0(\mcm)$ of certain bundles on $\mcm$. When, in addition, $\gcd(n_0,d_0)=1$, this open subset is in fact an irreducible component of ${\mc M}_0(\mcm)$. Results are also obtained for $n=1$; these require an extension of the arguments of \cite{k1} and \cite{mu}.

Our most important new results on the stability of Picard sheaves are as follows (full statements of most of these results may be found in section \ref{statement}).
Let $\theta_{n,d}$ and $\theta_{n,\xi}$ be theta divisors
on $\mcm$ and $\mcmx$ respectively. Before stating our first theorem, we define, for any $E\in{\mc M}_{n,d+n}$, $L\in\pic^{d'+1}(C)$, morphisms
\[\phi_E: C\longrightarrow\mcm:\ \ p\longmapsto E(-p),\]
\[\alpha_{L,E}: \pic^{d'}(C)\longrightarrow\mcm:\ \ L_1\longmapsto E\otimes L^{-1}\otimes L_1.\]
These are analogous to the classical Abel-Jacobi map embedding $C$ in $\picd$. Note that
\begin{equation}\label{eq0}
\phi_E=\alpha_{L,E}\circ\phi_L.
\end{equation}

\begin{theorem}\label{th13} Let $E_0$ be a stable bundle of rank $n_0$ and degree $d_0$. If $\gcd(n,d)=1$ and
$nd_0+n_0d> n_0n(2g-1)$ (respectively, $\ge$), then, for any $E\in{\mc M}_{n,d+n}$ for which
$E_0\otimes E$ is stable, and any $L\in\pic^{d'+1}(C)$,
\begin{itemize} 
\item[(i)] $\phi _E^*\mcw(E_0)$ is stable (respectively, semistable);
\item[(ii)] $\alpha_{L,E}^*\mcw(E_0)$ is $\theta_{1,d'}$-stable (respectively, semistable).
 \end{itemize}
 This holds in particular for $n_0=1$ and for $n_0\ge2$ with $E\in{\mc M}_{n,d+n}$ general.
 \end{theorem}
 This theorem is proved in section \ref{butler} as Theorem \ref{c1}(b).
 
 Our next theorem concerns Picard sheaves on $\mcm$ when $n=1$ (see Theorem \ref{th1}(i)). 

\begin{theorem}\label{th11} Let $L_0$ be a line bundle of degree $d_0$ on $C$.
If $d_0+d\ge g$ and $C$ is general, then $\mcwd(L_0)$ is $\theta_{1,d}$-stable.
\end{theorem}

The case $d_0+d\ge2g-1$ is classical (note that ${\mc M}_{1,d}=\pic^d(C)$), even without the assumption that $C$ is general; the corresponding result for bundles $E_0$ of higher rank is also known (see Theorem \ref{th1}(ii)).

For $\mcmx$, a strong result is already known when $E_0$ is a line bundle (see Theorem \ref{th2}(i) for details). To generalise this for $n_0\ge2$, we need to extend the concept of $(\ell,m)$-stability introduced in \cite{NR0, NR} to torsion-free sheaves on $\mcmx$ when $\gcd(n,d)=1$ (see section \ref{mcmx} for details).

\begin{theorem}\label{th14} (Theorem \ref{th2}(ii)) Let $n\ge2$, $\gcd(n,d)=1$ and let $E_0$ be a stable bundle of rank $n_0\ge2$ and degree $d_0$ on $C$. If either $nd_0+n_0d>n_0n(2g-2)$ or $E_0$ 
is general and $nd_0+n_0d>n_0ng-n_0$, then $\mcwx(E_0)$ is $(0,-n_0+1)$-$\theta_{n,\xi}$-stable.
\end{theorem}

Our final result on stability is Theorem \ref{th3}(ii).  

\begin{theorem}\label{th12} Let $n\ge2$, $\gcd(n,d)=1$ and let $E_0$ be a stable bundle of rank $n_0\ge2$ and degree $d_0$ on $C$. If $nd_0+n_0d> n_0n(n+1)(g-1)+n_0$ (respectively, $\ge$),
then $\mcw(E_0)$ is $\theta_{n,d}$-stable (respectively, semistable).
\end{theorem}

When $E_0$ is a line bundle, a stronger form of this theorem is known, which we extend further in Theorem \ref{th3}(i).

We turn now to the consideration of deformations of Picard bundles. Suppose that $\gcd (n,d)=1$ and $nd_0+n_0d>n_0n(2g-2)$. Denote by $\ch$ the Chern character of the Picard bundles $\mcw(E_0)$ (respectively, $\mcwx(E_0)$). With a few possible exceptions for low genus and rank, the Picard bundle $\mcw(E_0)$ (respectively, $\mcwx(E_0)$) is simple for all $E_0\in\mcmo$ (see Lemma \ref{l11} and Corollary \ref{c2}, respectively \cite[Corollary 21]{bisbn1}). The formula $(L,E_0)\mapsto L\otimes \mcw(E_0)$ (respectively, $E_0\mapsto\mcwx(E_0)$) therefore defines a morphism
\[\beta:\pic^0(\mcm)\times\mcmo\lra{\mc M}_0(\mcm)\ \ \textrm{(respectively, }\beta_\xi:\mcmo\lra{\mc M}_0(\mcmx)),\]
where $\pic^0(\mcm)$ denotes the space of topologically trivial line bundles on $\mcm$ and ${\mc M}_0(\mcm)$ (respectively, ${\mc M}_0(\mcmx)$) is the moduli space of simple bundles on $\mcm$ (respectively, $\mcmx$) with Chern character $\ch$. Note that $\pic^0(\mcm)\cong\pico$.
In order to obtain good properties of these morphisms, it is necessary to study the deformations of Picard bundles. As already stated, this study was initiated by Kempf \cite{k1} and Mukai \cite{mu} in the case where $n=n_0=1$. Deformations of Picard bundles on ${\mcmx}$ were studied in \cite{bisbn1}. The first author and Ravindra, using Mukai's techniques, showed that, when $n=1$ and $\xi_0$ is any line bundle of degree $d_0$ on $C$, $\beta|\{{\mc O}_{\mcm}\}\times{\mc M}_{n_0,\xi_0}$ is an injective morphism \cite[Corollary 2.3]{br}.  Here we consider Picard bundles on ${\mcm}$ and prove in particular the following theorem (see Theorem \ref{th4}).

\begin{theorem}\label{th15} Suppose that
\[g\ge3,\ n\ge2,\ \gcd(n,d)=\gcd(n_0,d_0)=1,\ nd_0+n_0d>n_0n(2g-2)\]
and suppose further that, if $g=3$, then $n\ge4$ and, if $g=4$, then $n\ge3$. Then $\beta$ maps $\pic^0(\mcm)\times\mcmo$ isomorphically onto an irreducible component ${\mc M}_0^0(\mcm)$ of 
${\mc M}_0(\mcm)$ of dimension $g+n_0^2(g-1)+1$; hence ${\mc M}_0^0(\mcm)$ is isomorphic to $\pico\times\mcmo$ and is a fine moduli space for deformations of Picard bundles on $\mcm$ with Chern character $\ch$. If $nd_0+n_0d>n_0n(n+1)(g-1)+n_0$, then ${\mc M}_0^0(\mcm)$ is a component of the moduli space of $\theta_{n,d}$-stable bundles on $\mcm$ with Chern character $\ch$.
\end{theorem}

This theorem needs modifying when $n=1$. Let $N$ denote the bundle defined by the evaluation sequence
\[0\lra N^*\lra H^0(C,K_C)\otimes {\mc O}_C\stackrel{\operatorname{ev}}{\lra} K_C\lra0.\]
 The following theorem is a summary of Theorem \ref{th7}.       
 \begin{theorem}\label{th16}
 Suppose that $g\ge2$, $n_0\ge2$ and $d_0+n_0d>n_0(2g-2)$. Then the morphism $\beta:\pic({\mc M}_{1,d})\times\mcmo\to{\mc M}_0({\mc M}_{1,d})$ is injective. Moreover, if $C$ is non-hyperelliptic and $h^0(C,N\otimes\operatorname{ad}(E_0))=0$ for some $E_0\in\mcmo$, $\beta$ is an injective birational morphism from $\pic^0({\mc M}_{1,d})\times\mcmo$ to an irreducible component ${\mc M}_0^0({\mc M}_{1,d})$ of ${\mc M}_0({\mc M}_{1,d})$. If, in addition, $\gcd(n_0,d_0)=1$, $\beta$ is a bijective morphism onto ${\mc M}_0^0({\mc M}_{1,d})$.
\end{theorem}

It is of interest to note that, when $C$ is non-hyperelliptic, the condition $h^0(C,N\otimes\operatorname{ad}(E_0))=0$ is equivalent to the vanishing of the Koszul cohomology group 
\[K_{0,2}(C;\operatorname{ad}(E_0),K_C)\] 
(see Proposition \ref{prop5}). This expresses the surjectivity of the multiplication map
\[\mu_{E_0}:H^0(C,K_C)\otimes H^0(C,K_C\otimes E_0\otimes E_0^*)\lra H^0(C,K_C^2\otimes E_0\otimes E_0^*).\]

We show also that, at least for some values of $n_0$, $d_0$, $\beta$ does not map $\pic^0({\mc M}_{1,d})\times\mcmo$ isomorphically to ${\mc M}_0^0(\mcm)$ (Proposition \ref{prop2}). On the other hand, for general $C$ and general $E_0$, Montserrat Teixidor i Bigas has proved that $\mu_{E_0}$ is surjective \cite{t}; this leads to Corollary \ref{th8}.

The main results on stability, together with comments on the tools used, are stated in section \ref{statement}. In section \ref{butler}, we relate our problem to a conjecture of D. C. Butler; this result (Proposition \ref{prop0}) seems to be of interest in its own right and leads to an important result on the stability of $\phi_E^*(\mcw(E_0))$ and $\alpha^*_{L,E}(\mcw(E_0))$ (Theorem \ref{c1}). Sections \ref{picd}--\ref{mcm} are devoted to the Picard sheaves on the various moduli spaces. In section \ref{defo}, we consider deformations of Picard bundles on $\mcm$ for $n\ge2$ and prove a more precise version of Theorem \ref{th15} (Theorem \ref{th4}) covering also the case $\gcd(n_0,d_0)\ne1$. In section 8, we study deformations of Picard bundles on ${\mc M}_{1,d}$. We start with a statement of the main results of Kempf \cite{k1} and Mukai \cite{mu} for the case $n_0=1$ (Theorem \ref{th6}) and finish with a more detailed version of Theorem \ref{th16} (Theorem \ref{th7}) and proofs of Proposition \ref{prop2} and Corollary \ref{th8}. 

We assume throughout that $C$ is a smooth irreducible projective curve over ${\mathbb C}$ of genus $g\ge2$. The canonical line bundle on $C$ is denoted by $K_C$. For any coherent sheaf $F$ on $C$, we write $H^i(F)$ for $H^i(C,F)$ and $h^i(F)$ for the dimension of $H^i(C,F)$. For any sheaf $F$ on a scheme $X$ and any $p\in X$, we write $F_p$ for the fibre of $F$ at $p$. We also write $\ch$ for the Chern character of $\mcw(E_0)$ (or $\mcwx(E_0)$); $\ch$ depends only on $n_0$, $d_0$, $n$, $d$ and $\mcu$ (or $\mcux$) and these values will be clear from the context.

Our thanks are due to the referee of this paper and to the referee of a previous version for some useful comments. We thank also Montserrat Teixidor i Bigas for discussions on the surjectivity of $\mu_{E_0}$.

\section{Statement of results on stability of Picard sheaves}\label{statement}

In this section, we state the main theorems on stability of Picard sheaves and comment on the tools used to prove them. We include all the results of which we have knowledge and state in detail which are already known and which are new. Proofs will be given in sections 4 to 6.

\begin{theorem}\label{th1} Let $E_0$ be a vector bundle of rank $n_0$ and degree $d_0$ on $C$.
\begin{itemize}
\item[(i)] If $n_0=1$ and either $d_0+d\ge 2g-1$ or $d_0+d\ge g$ and $C$ is general, then $\mcwd(E_0)$ is $\theta_{1,d}$-stable.
\item[(ii)] If $n_0\ge2$, $E_0$ is stable and $d_0+n_0d> n_0(2g-1)$ (respectively, $\ge$), then $\mcwd(E_0)$ is $\theta_{1,d}$-stable (respectively, semistable).
\end{itemize}\end{theorem}

When 
$E_0=\mco$, Theorem \ref{th1}(i) is known for $d\ge2g-1$ \cite{k,el}.  Alternative proofs of (ii) are available \cite{br, hp}.

In the next theorem, we need the concept of stability for projective bundles. In fact, a projective bundle is $\theta$-stable if and only if the associated principal $\operatorname{PGL}$-bundle is $\theta$-stable. For further discussion of this, in our context, see \cite{bisbn2}.

\begin{theorem}\label{th2} Let $n\ge2$ and let $E_0$ be a vector bundle of rank $n_0$ and degree $d_0$ on $C$.
\begin{itemize}
\item[(i)] If $n_0=1$, $nd_0+d>n(g-1)$ and either $g\ge3$ or $g=2$ and $d$ is not a multiple of $n$, then $\mcwxp(E_0)$ is $\theta_{n,\xi}$-stable. Moreover, if in addition $\gcd(n,d)=1$, then $\mcwx(E_0)$ is $\theta_{n,\xi}$-stable.
\item[(ii)]  If $n_0\ge2$, $\gcd(n,d)=1$, $E_0$ is stable and either $nd_0+n_0d>n_0n(2g-2)$ or $E_0$ 
is general and $nd_0+n_0d>n_0ng-n_0$, then $\mcwx(E_0)$ is $(0,-n_0+1)$-$\theta_{n,\xi}$-stable.  
\end{itemize}\end {theorem}

If $n_0=1$ and $nd_0+d\le n(g-1)$, then $\mcwxp(E_0)=\emptyset$ and, when, in addition, $\gcd(n,d)=1$,
we have $\mcwx(E_0)=0$. 
So Theorem \ref{th2}(i) is best possible for $g\ge3$ (in which case it is known \cite{bisbn2}) and almost best possible for $g=2$. If 
$\gcd(n,d)=1$, Theorem \ref{th2}(i) was also proved in \cite{bbgn} when $nd_0+d>2n(g-1)$. The new results are (i) for $g=2$ and (ii); note that (ii) does not imply that $\mcwx(E_0)$ is $\theta_{n,\xi}$-stable.

\begin{theorem}\label{th3} Let $n\ge2$, $\gcd(n,d)=1$ and let $E_0$ be a vector bundle of rank $n_0$ and degree $d_0$ on $C$.
\begin{itemize}
\item[(i)] If $n_0=1$ and $nd_0+d\ge n(2g-1)$, then $\mcw(E_0)$ is $\theta_{n,d}$-stable.
\item[(ii)] If $n_0\ge2$, $E_0$ is stable and $nd_0+n_0d> n_0n(n+1)(g-1)+n_0$ (respectively, $\ge$),
then $\mcw(E_0)$ is $\theta_{n,d}$-stable (respectively, semistable).
\end{itemize}\end{theorem}

Theorem \ref{th3}(i) is known 
for $nd_0+d>2ng$ \cite{li}; (ii) is new.

 The main tools for the
proofs of the above results are the generalisation, and adaptation to our cases, of Lemmas 1.1 and 1.2 in \cite{el} and \cite[Lemma 2.8(2)]{li}, and the use of Hecke correspondences as in \cite{bbgn, bisbn2}. The first tool will be developed in Section \ref{butler}, where we relate our problem to a conjecture of D. C. Butler; no further tools are needed in Section \ref{picd}. Hecke correspondences are used in Section \ref{mcmx} and spectral curves in Section \ref{mcm}. 

\section{Picard sheaves and Butler's Conjecture}\label{butler}

In this section, we relate Picard sheaves on $\mcm$ to a conjecture of D. C. Butler \cite[Conjecture 2]{bu2}. Butler's conjecture is concerned with the following construction. Given a generated vector bundle $E$ of rank $n \geq 1$ and degree $d$ on $C$, we define a vector bundle $M_E$ by the exact sequence
\begin{equation}\label{eq1}
0\lra M_E\lra H^0(E)\otimes {\mc O}\lra E\lra 0.
\end{equation}
In \cite[Theorem 1.2]{bu}, Butler proved that, if $E$ is stable of degree $d>2ng$, then $M_E$ is stable. His conjecture (see \cite[Conjecture 2]{bu2}) is a generalisation of this. The form of this conjecture that is relevant for us asserts that, on a general curve, the bundle $M_E$ is stable for general stable $E$ of any degree. The following proposition is a generalisation of part of \cite[Lemma 1.1]{el} and of a result proved but not formally stated in \cite[p. 536]{li} and links Butler's Conjecture to our problem. The proof follows the same lines as that of \cite[Theorem 2.5]{li}.

\begin{proposition}\label{prop0}
Suppose that $\gcd(n,d)=1$. Let $E\in{\mc M}_{n,d+n}$ and let $E_0$ be a vector bundle on $C$ such that
$E_0\otimes E$ is stable and generated with $h^1(E_0\otimes E)=0$. Then
\[\phi^*_E(\mcw(E_0))\cong M_{E_0\otimes E}\otimes L'\]
for some line bundle $L'$ on $C$.
\end{proposition}

\begin{proof}
Let $\Delta$ be the diagonal of $C\times C$. The vector bundles $(\phi_E\times 1_C)^*(p_2^*(E_0)\otimes \mcu)$ and $p_2^*(E_0\otimes E)(-\Delta)$ coincide as families of stable bundles on $C$ with respect to $p_1$. It follows that there exists a line bundle $L'$ on $C$ such that
\begin{equation}\label{eq2}
p_2^*(E_0\otimes E)(-\Delta)\otimes p_1^*(L')\cong (\phi_E\times 1_C)^*(p_2^*(E_0)\otimes \mcu).
\end{equation}
Tensoring the exact sequence $$0\longrightarrow{\mc O}_{C\times C}(-\Delta)\longrightarrow
{\mc O}_{C\times C}\longrightarrow{\mc O}_\Delta\longrightarrow 0$$
by $p_2^*(E_0\otimes E)$ and taking direct images by $p_1$, we obtain an exact sequence
\begin{equation}\label{eq3}
0\lra p_{1*}(p_2^*(E_0\otimes E)(-\Delta))\lra H^0(E_0\otimes E)\otimes{\mco}\lra E_0\otimes E.
\end{equation}
Since $E_0\otimes E$ is generated, the right-hand map in \eqref{eq3} is surjective, so $M_{E_0\otimes E}\cong p_{1*}(p_2^*(E_0\otimes E)(-\Delta))$. Hence, by \eqref{eq2},
\begin{equation}\label{eq5}
M_{E_0\otimes E}\otimes L'\cong p_{1*}(\phi_E\times 1_C)^*(p_2^*(E_0)\otimes \mcu).
\end{equation}
Since $E_0\otimes E$ is generated and $h^1(E_0\otimes E)=0$, we have $h^1(E_0\otimes E)(-p)=0$ for all $p\in C$. Hence, by \cite[p.53, Corollary 3]{mf}, 
\[p_{1*}(\phi_E\times 1_C)^*(p_2^*(E_0)\otimes \mcu)\cong \phi_E^*(p_{1*}(p_2^*(E_0)\otimes\mcu))=\phi_E^*(\mcw(E_0)).\]
The result now follows from \eqref{eq5}.
\end{proof}

This leads in particular to the main theorem of this section. 

\begin{theorem}\label{c1}
Suppose that $\gcd(n,d)=1$, $E\in{\mc M}_{n,d+n}$, $L\in\pic^{d'+1}$ and $E_0\in{\mc M}_{n_0,d_0}$.

{\em (a)} If $E_0\otimes E$ is stable and generated with $h^1(E_0\otimes E)=0$ and $M_{E_0\otimes E}$ is stable (respectively, semistable), then
\begin{itemize}
\item[(i)] $\phi_E^*(\mcw(E_0))$ is stable (respectively, semistable);
\item[(ii)] $\alpha_{L,E}^*(\mcw(E_0))$ is $\theta_{1,d'}$-stable (respectively, semistable).
\end{itemize}

{\em (b)} If $nd_0+n_0d>n_0n(2g-1)$ (respectively, $\ge$), then (i) and (ii) hold for any $E\in{\mc M}_{n,d+n}$ for which $E_0\otimes E$ is stable, and in particular for general $E\in{\mc M}_{n,d+n}$.
\end{theorem}

To prove this theorem, we need some lemmas. The first is \cite[Lemma 2.7]{li}.

\begin{lemma}\label{l1}
Let $A$ be an abelian variety, $B$ and $C$ subvarieties of $A$ satisfying $\dim B+\dim C<\dim A$. Then the set $U:=\{t\in A|(B+t)\cap C=\emptyset\}$ is a non-empty open subset of $A$.
\end{lemma}

We use this lemma to generalise \cite[Lemma 1.2]{el} and \cite[Lemma 2.8(2)]{li} to torsion-free sheaves.

\begin{lemma}\label{l2}
Let $\mc E$ be a torsion-free sheaf on $\pic^{d'}(C)$. If $\phi^*_L(\mc E)$ is stable (respectively, semistable) for some $L\in\pic^{d'+1}(C)$, then $\mc E$ is $\theta_{1,d'}$-stable (respectively, semistable).
\end{lemma}
\begin{proof}
Let ${\mc F}$ be a proper torsion-free subsheaf of ${\mc E}$ such that ${\mc E}/{\mc F}$ is also torsion-free. The set of points of $\pic^{d'}(C)$ at which at least one of ${\mc E}$, $\mc F$ and ${\mc E}/{\mc F}$ fails to be locally free is a closed subset of codimension at least $2$. It follows from Lemma \ref{l1} that there is an open set $U\subset\pic^{d'+1}(C)$ such that, for all $L\in U$, $\phi^*_L(\mc E)$ is a vector bundle and $\phi^*_L(\mc F)$ is a proper subbundle. Since stability is an open condition, the hypotheses of the lemma allow us to assume that $\phi^*_L(\mc E)$ is also stable. Hence
\[c_1(\phi^*_L(\mc F))/\rk(\phi^*_L(\mc F))<c_1(\phi^*_L(\mc E))/\rk(\phi^*_L(\mc E)),\]
or, equivalently,
\[\phi_L(C)\cdot c_1({\mc F})/\rk(\mc F)<\phi_L(C)\cdot c_1({\mc E})/\rk(\mc E).\]
Since $\phi_L(C)$ is cohomologically equivalent to $c_1(\theta_{1,d'})^{g-1}/(g-1)!$ by the Poincar\'e formula, this is just the $\theta_{1,d'}$-stability condition for ${\mc E}$.

For the semistable version, we simply replace $<$ by $\le$.
\end{proof}

\begin{proof}[Proof of Theorem \ref{c1}] (a) (i) is immediate from Proposition \ref{prop0}. (ii) then follows from \eqref{eq0} and Lemma \ref{l2}.

(b)  By \cite[Theorem 3.10]{ct} (attributed to S. Ramanan), the bundle $E_0\otimes E$ is stable for general $E$. Moreover, if $nd_0+n_0d\ge n_0n(2g-1)$, then $E_0\otimes E$ is generated and $h^1(E_0\otimes E)=0$. It follows from \cite[Theorem 1.2]{bu} that $M_{E_0\otimes E}$ is semistable and is stable if $nd_0+n_0d>n_0n(2g-1)$. So (b) follows from (a).
\end{proof}

In order to apply Theorem \ref{c1} to our problem, we need to relate the stability of $\mcw(E_0)$ to that of $\phi_E^*(\mcw(E_0))$ and $\alpha_{L,E}^*(\mcw(E_0))$. It will turn out that this is easy when $n=1$ but more difficult for $n\ge2$.

\section{Picard sheaves on $\picd$}\label{picd}

In this section, the following propositions will prove Theorem \ref{th1}. 
Note that, when $n=1$, we can take $E=L\in\pic^{d+1}(C)$. It follows that $\alpha_{L,E}=1_{\picd}$.

\begin{proposition}\label{l4} Suppose that $E_0\in\mcmo$.
If $d_0+n_0d> n_0(2g-1)$ (respectively, $\ge$), then $\mcwd(E_0)$ is $\theta_{1,d}$-stable (respectively, semistable).
\end{proposition}

\begin{proof}
In this case, $E_0\otimes L$ is certainly stable. The result follows at once from Theorem \ref{c1}(b) and the fact that $\alpha_{L,L}=1_{\picd}$.
\end{proof}

\begin{remark}\label{r0}\begin{em}
If $d_0+d=2g-1$, it is in fact true that $\mcwd(L_0)$ is $\theta_{1,d}$-stable. This follows from \cite{k}, where the result is proved for $L_0=\mco$.
\end{em}\end{remark}

\begin{proposition}\label{l5}
Let $C$ be a general curve of genus $g\ge2$ and $L_0$ a line bundle on $C$ of degree $d_0$ with $d_0+d\ge g$. Then $\mcwd(L_0)$ is $\theta_{1,d}$-stable.
\end{proposition}
\begin{proof}
For general $L\in\pic^{d+1}(C)$, we have $h^1(L_0\otimes L)=0$ and $L_0\otimes L$ is generated. Moreover, for $g\ge3$, $M_{L_0\otimes L}$ is stable by \cite[Theorem 2]{bu2} (see also \cite[Proposition 4.1]{brp}) and the result follows from Theorem \ref{c1}(a)(ii). For $g=2$, by Remark \ref{r0}, the only outstanding case is $d_0+d=2$ and then $\mcwd(L_0)$ has rank $1$.
\end{proof}

Propositions \ref{l4} and \ref{l5} and Remark \ref{r0} complete the proof of Theorem \ref{th1}.

\begin{remark}\label{r1}\begin{em}
For other proofs of Theorem \ref{th1}(ii), see \cite[Lemma 2.1]{br} or \cite[Theorem A]{hp}; a result on semistability when $n_0(g-1)<d_0\le n_0g$ may be found in \cite[Theorem B]{hp}.
\end{em}\end{remark}

\section{Picard sheaves on $\mcmx$}\label{mcmx}

In this section, we are concerned with results for Picard sheaves on $\mcmx$ and, in particular, with 
establishing Theorem \ref{th2}. For Theorem \ref{th2}(i), we are not assuming that 
$\gcd(n,d)=1$, so we need to show that $\theta_{n,\xi}$-stability is well defined on $\mcmx$ and on 
the open subset $\mcmxxl\,\subset\,\mcmx$ (see \eqref{e1}).
Recall first that $\mcmx$ has a natural compactification $\overline{\mc M}_{n,\xi}$, which is locally factorial with $\pic(\overline{\mc M}_{n,\xi})\cong\mathbb{Z}$ \cite{dn}. Unless $g=n=2$ and $d$ is even, the complement of $\mcmx$ in $\overline{\mc M}_{n,\xi}$ coincides with the singular set of $\overline{\mc M}_{n,\xi}$ \cite[Theorem 1]{NR1} and therefore has codimension $\ge2$. Except in this case, we therefore have $\pic(\mcmx)
\,\cong\,
{\mathbb Z}$ and we can take $\theta_{n,\xi}$ to be the positive generator. 

In order to prove Theorem \ref{th2}(i), we need some lemmas.

\begin{lemma}\label{l12}
If $L_0$ is a line bundle of degree $d_0$ and $nd_0+d>n(g-1)$, then, unless $g=n=2$ and $d$ is even, the complement of $\mcmxxl$ in $\overline{\mc M}_{n,\xi}$ has codimension $\ge2$.
\end{lemma}

\begin{proof}
 Since we know that, under the hypotheses of the lemma, the complement of $\mcmx$ in $\overline{\mc M}_{n,\xi}$ has codimension $\ge2$, it remains to prove that the complement of $\mcmxxl$ in $\mcmx$ has codimension $\ge2$. For $g\ge3$, this is proved in \cite[Lemma 4.1]{bisbn2}. In fact, the proof of that lemma shows that the codimension $\ge 1+nd_0+d-n(g-1)$ whenever $g\ge2$, which gives the required result.
\end{proof}

Under the hypotheses of Lemma \ref{l12}, we now see that $\pic(\mcmxxl)\cong{\mathbb Z}$ and that the restriction of the positive generator $\theta_{n,\xi}$ of $\pic(\mcmx)$ to $\mcmxxl$ generates $\pic(\mcmxxl)$; we continue to denote this generator by $\theta_{n,\xi}$. We can therefore extend the concept of $\theta_{n,\xi}$-stability to torsion-free sheaves and projective bundles on $\mcmxxl$.

Recall that a vector bundle $F$ on $C$ is $(\ell,m)$-stable if 
\[\frac{\deg F'+\ell}{\rk F'}<\frac{\deg F+\ell-m}{\rk F}\]
for every proper subbundle $F'$ of $F$ (see \cite{NR0,NR}).

\begin{lemma}\label{l14}
Let $\xi$ be a line bundle of degree $d$ on $C$, $L_0$ a line bundle of degree $d_0$ and $p\in C$. Then there exist $(0,1)$-stable bundles of rank $n$ and determinant $L_0^n\otimes\xi(p)$ if and only if either $g\ge3$ or $g=2$ and $d$ is not a multiple of $n$.
\end{lemma}

\begin{proof}
When $\gcd(n,d)=1$, the existence of $(0,1)$-stable bundles in ${\mc M}_{L_0^n\otimes\xi(p)}$ follows from \cite[Lemma 2]{bbgn}. In fact, the proof of that lemma shows that $(0,1)$-stable bundles exist unless $g=2$ and there exists an integer $e$ such that $ne=(n-1)(nd_0+d)$, in other words, $d$ is a multiple of $n$. It remains to show that, if $g=2$ and $F\in{\mc M}_{L_0^n\otimes\xi(p)}$ with $d$ a multiple of $n$, then $F$ is not $(0,1)$-stable. In fact, by \cite{ms}, any vector bundle $F$ of rank $n$ and degree $nd_0+d+1$ admits a subbundle of rank $n-1$ and degree $d'$ with 
\[(n-1)(nd_0+d+1)-nd'\le(n-1)g=2(n-1).\]
This condition simplifies to $nd'\ge(n-1)(nd_0+d-1)$. Since $d$ is a multiple of $n$, this is equivalent to 
\[\frac{d'}{n-1}\ge\frac{nd_0+d}n,\] which contradicts the $(0,1)$-stability of $F$.
\end{proof}

\begin{proof}[Proof of Theorem \ref{th2}(i)]
Let $L_0$ be a line bundle of degree $d_0$ with $nd_0+d>n(g-1)$ and let $f:\mcmx\longrightarrow{\mc M}_{n,L_0^n\otimes\xi}$ be defined by $f(E)=L_0\otimes E$. Then
\[f^*(\mcpw_{n,L_0^n\otimes\xi}(\mco))\,\cong\, \mcwxp(L_0)\, .\]
If $g\ge3$, Theorem \ref{th2}(i) now follows directly from \cite[Theorem 4.4 and Corollary 4.5]{bisbn2}. When $g=2$, we use Lemma \ref{l12} in place of \cite[Lemma 4.1]{bisbn2} and Lemma \ref{l14} in place of \cite[Lemma 3.4]{bisbn2}. The proofs of \cite[Theorem 4.4 and Corollary 4.5]{bisbn2} now remain valid.\end{proof}

We turn to the case $n_0\ge2$ and assume that $\gcd(n,d)=1$. Now $\mcmx$ is a smooth projective variety with $\pic(\mcmx)\cong {\mathbb Z}$ and the Picard sheaf is defined on the whole of $\mcmx$. We shall  need a generalisation of the concept of $(\ell,m)$-stability to torsion-free sheaves on $\mcmx$. For any such sheaf ${\mc E}$, we can write $c_1(\mc E)=\lambda_{\mc E} c_1(\theta_{n,\xi})$ for some integer $\lambda_{\mc E}$. The sheaf $\mc E$ is now $\theta_{n,\xi}$-stable (semistable) if and only if, for every proper subsheaf $\mc F$ of $\mc E$,
\[\frac{\lambda_{\mc F}}{\rk {\mc F}}<(\le)\frac{\lambda_{\mc E}}{\rk {\mc E}}.\]
\begin{definition}\label{def}\begin{em}
Suppose that $g\ge2$ and $\gcd(n,d)=1$. A torsion-free sheaf $\mc E$ on $\mcmx$ is $(\ell,m)$-$\theta_{n,\xi}$-\textit{stable (semistable)} if, for every proper subsheaf $\mc F$ of $\mc E$,
\begin{equation}\label{eq11}
\frac{\lambda_{\mc F}+\ell}{\rk {\mc F}}<(\le)\frac{\lambda_{\mc E}+\ell-m}{\rk {\mc E}}.
\end{equation}
\end{em}\end{definition}

This definition makes sense on any quasi-projective variety whose Picard group is isomorphic to ${\mathbb Z}$.

We now recall more details from \cite{bbgn}. For any vector bundle $F$ of rank $n$ and determinant $\xi(p)$ with $p\in C$, the non-trivial exact sequences
\begin{equation}\label{eq8}
0\lra E\lra F\lra {\mathbb C}_p\lra0
\end{equation}
form a family parametrised by the projective space ${\mathbb P}(F_p^*)$. If $F$ is $(0,1)$-stable, then $E$ is stable, so we obtain a morphism
\[\psi_{F,p}:{\mathbb P}(F_p^*)\lra\mcmx.\]

\begin{lemma}\label{l18}
Suppose that $\gcd(n,d)=1$ and let $F$ be a $(0,1)$-stable bundle of rank $n$ and determinant $\xi(p)$ for some line bundle $\xi$ of degree $d$ and some $p\in C$. Then $\psi_{F,p}$ is an isomorphism onto its image and 
\begin{equation}\label{eq12}
\psi_{F,p}^*(\theta_{n,\xi})\cong{\mc O}_{{\mathbb P}(F_p^*)}(1).
\end{equation}
\end{lemma}

\begin{proof}
For the first statement, see \cite[Lemma 5.9]{NR} (or \cite[Lemma 3]{bbgn}). After tensoring by a line bundle on $C$, we can suppose that $d>2n(g-1)$. It follows from \cite[Diagram (6)]{bbgn} that, for the integer $j$ defined in \cite[Formula (3)]{bbgn}, there is an exact sequence
\[0\lra\psi^*_{F,p}(\mcwx({\mc O}_C)(-j))\lra H^0(F)\otimes{\mc O}_{{\mathbb P}(F_p^*)}\lra{\mc O}_{{\mathbb P}(F_p^*)}(1)\lra0.\]
Hence $\psi^*_{F,p}(\mcwx({\mc O}_C)(-j))$ has degree $-1$. Since $\theta_{n,\xi}$ is the positive generator of $\pic(\mcmx)$, the formula \eqref{eq12} follows.
\end{proof}

In view of Lemma \ref{l18}, we can identify ${\mathbb P}(F_p^*)$ with its image in $\mcmx$.

\begin{lemma}\label{l19}
Suppose that $\gcd(n,d)=1$, $E_0\in{\mc M}_{n_0,d_0}$ and $F\in{\mc M}_{n,\xi(p)}$. Suppose further that one of the following holds:
\begin{itemize}
\item[(i)] $nd_0+n_0d>n_0n(2g-2)$ and $F$ is $(0,1)$-stable;
\item[(ii)] $nd_0+n_0d>n_0ng-n_0$ and $E_0$ and $F$ are general.
\end{itemize} 
Then there exists an exact sequence
\begin{eqnarray}\label{eq9}
\nonumber0\lra H^0(E_0\otimes F(-p))\otimes{\mc O}_{{\mathbb P}(F_p^*)}&\lra&\psi_{F,p}^*(\mcwx(E_0))(-j)\lra\Omega_{{\mathbb P}(F_p^*)}(1)\otimes(E_0)_p\\&\lra& H^1(E_0\otimes F(-p))\otimes{\mc O}_{{\mathbb P}(F_p^*)}\lra0.
\end{eqnarray}
\end{lemma}

\begin{proof}
The bundle $E_0\otimes E$ is semistable for every $E\in\mcmx$. Hence, if (i) holds, $H^1(E_0\otimes E)=0$ for all such $E$. Tensoring by $E_0$ in \cite[Diagram (4)]{bbgn} and by $p_2^*(E_0)$ in \cite[Diagram (5)]{bbgn} (note that our $p_1$, $p_2$ correspond respectively to $p_2$, $p_1$ in \cite{bbgn}), we obtain from \cite[Diagram (6)]{bbgn} the required exact sequence \eqref{eq9}.

Now suppose that (ii) holds. The bundle $E_0\otimes F(-p)$ is semistable. If $L$ is a general element of $\pico$, then $F(-p)\otimes L$ is a general element of ${\mc M}_{n,d+1-n}$; moreover, for any $L$, $E_0\otimes L^{-1}$ is a general element of ${\mc M}_{n_0,d_0}$. It follows from \cite[Theorem 4.6]{h} that 
\[E_0\otimes L^{-1}\otimes F(-p)\otimes L=E_0\otimes F(-p)\]
is non-special. Since $nd_0+n_0(d+1-n)>n_0n(g-1)$, this implies that $H^1(E_0\otimes F(-p))=0$.
It follows that $H^1(E_0\otimes E)=0$ for all $E\in {\mathbb P}(F_p^*)$. The argument is completed as in case (i).
\end{proof}

\begin{remark}\label{r8}\begin{em}If the hypotheses of Lemma \ref{l19} hold, then \eqref{eq9} implies that $\psi_{F,p}^*(\mcwx(E_0))$ is locally free. Moreover, since $\Omega_{{\mathbb P}(F_p^*)}(1)$ has degree $-1$ and $E_0$ has rank $n_0$, the bundle $\Omega_{{\mathbb P}(F_p^*)}(1)\otimes(E_0)_p$ has degree $-n_0$. It follows at once from \eqref{eq9} that $\psi_{F,p}^*(\mcwx(E_0))(-j)$ also has degree $-n_0$; so $\psi_{F,p}^*(\mcwx(E_0))$ is not semistable.
\end{em}\end{remark}

\begin{lemma}\label{l15}
Suppose that the hypotheses of Lemma \ref{l19} hold. Then, for any proper subsheaf $\mc G$ of rank $r$ of $\psi_{F,p}^*(\mcwx(E_0))(-j)$ whose image in $\Omega_{{\mathbb P}(F_p^*)}(1)\otimes(E_0)_p$ is non-zero,
\begin{equation}\label{eq10}
\frac{\deg {\mc G}}r<\frac{\deg\psi_{F,p}^*(\mcwx(E_0))(-j)+n_0-1}{\rk\mcwx(E_0)}.
\end{equation}.
\end{lemma}

\begin{proof} 
Since $\Omega_{{\mathbb P}(F_p^*)}(1)\otimes(E_0)_p$ is semistable of negative degree, it follows from the hypothesis and \eqref{eq9} that $\deg{\mc G}\le{-1}$. So
\[\frac{\deg {\mc G}}r\le\frac{-1}r<\frac{-1}{\rk\mcwx(E_0)}=\frac{\deg\psi_{F,p}^*(\mcwx(E_0))(-j)+n_0-1}{\rk\mcwx(E_0)},\]
since $\deg\psi_{F,p}^*(\mcwx(E_0))(-j)= -n_0$ (see Remark \ref{r8}).
\end{proof}

\begin{lemma}\label{l16}
Let $p_1,\ldots,p_m\in C$ and let $\mc F$ be a subsheaf of $\mcwx(E_0)$. There exists a non-empty open subset $U$ of $\mcmx$ such that, if $E\in U$, then
\begin{itemize}
\item[(i)] $\mc F$ is locally free at $E$;
\item[(ii)] the homomorphism of fibres ${\mc F}_E\longrightarrow \mcwx(E_0)_E$ is injective;
\item[(iii)] for all $p_i$ and for the generic extension \eqref{eq8} with $p=p_i$, the vector bundle $F$ is $(0,1)$-stable and $\mc F$ is locally free at every point of $\psi_{F,p_i}({\mathbb P}(F_{p_i}^*))$ outside some subvariety of codimension at least $2$.
\end{itemize}
\end{lemma}

\begin{proof} The proof is identical with that of \cite[Lemma 4]{bbgn}.
\end{proof}

\begin{lemma}\label{l17}
Suppose that $\gcd(n,d)=1$, $E_0\in{\mc M}_{n_0,d_0}$ and that either $nd_0+n_0d>n_0n(2g-2)$ or $E_0$ 
is general and $nd_0+n_0d>n_0ng-n_0$. Let $\mc F$ be a proper subsheaf of $\mcwx(E_0)$. Then  there exist $p\in C$ and $F\in {\mc M}_{n,\xi}$ such
that the image of $\psi_{F,p}^*{\mc F}(-j)$ in $\Omega_{{\mathbb P}(F_{p}^*)}(1)\otimes(E_0)_p$ is non-zero.
\end{lemma}

\begin{proof}
We follow the proof on p.567 of \cite{bbgn}. Choose points $p_1,\ldots,p_m\in C$ with $m>\frac{nd_0+n_0d}{n_0n}$ and choose $E$ and $F$ as in Lemma \ref{l16}. In particular, $F$ is $(0,1)$-stable, so Lemma \ref{l18} applies and \eqref{eq9} holds. Since $E_0\otimes E$ is semistable, $H^0(E_0\otimes E(-p_1-\ldots-p_m))=0$. Let $v$ be a non-zero element of ${\mc F}_E$. By Lemma \ref{l16}, the image $s$ of $v$ in $\mcwx(E_0)_E$ is non-zero. Since $H^0(E_0\otimes E(-p_1-\ldots-p_m))=0$, there exists $p:=p_i$ such that $s(p)\ne0$. By further restricting $F$, we can suppose that $s\not\in H^0(E_0\otimes F(-p))$. The result now follows from \eqref{eq9}.
\end {proof}

\begin{proof}[Proof of Theorem \ref{th2}(ii)] Let $\mc F$ be a proper subsheaf of $\mcwx(E_0)$ of rank $r$. Choose $p$ and $F$ as in Lemma \ref{l17} and let ${\mc F}_1$ be the image of $\psi_{F,p}^*{\mc F}(-j)$ in $\psi_{F,p}^*(\mcwx(E_0))(-j)$. In view of Lemma \ref{l17}, we can take ${\mc G}={\mc F}_1$ in Lemma \ref{l15}. By Lemma \ref{l16}(ii), the homomorphism $\psi_{F,p}^*{\mc F}(-j)\to{\mc F}_1$ is an isomorphism in the neighbourhood of $E$. By Lemma \ref{l16}(iii), the kernel of this homomorphism is supported on a subvariety of codimension at least $2$. It follows that the homomorphism is an isomorphism away from this subvariety. It follows from Lemma \ref{l15} that 
\[\frac{\deg\psi_{F,p}^*{\mc F}(-j)}r<\frac{\deg\psi_{F,p}^*(\mcwx(E_0))(-j)+n_0-1}{\rk\mcwx(E_0)}.\]
By \eqref{eq12}, we have
\[\frac{\lambda_{\mc F}}r=\frac{\deg\psi_{F,p}^*{\mc F}}r<\frac{\deg\psi_{F,p}^*(\mcwx(E_0))+n_0-1}{\rk\mcwx(E_0)}=\frac{\lambda_{\mcwx(E_0)}+(n_0-1)}{\rk\mcwx(E_0)}.\]
This completes the proof.
\end{proof}

\section{Picard sheaves on $\mcm$}\label{mcm}

In this section, we prove Theorem \ref{th3}.

\begin{lemma}\label{l6}
Suppose that $\gcd(n,d)=1$ and let $L_0$ be a line bundle of degree $d_0$. If $nd_0+d\ge n(2g-1)$, then $\mcw(L_0)$ is $\theta_{n,d}$-stable.
\end{lemma}
\begin{proof} Suppose first that $nd_0+d>2ng$. Let $f':\mcm\stackrel{\cong}{\longrightarrow} {\mc M}_{n,nd_0+d}$ be defined by $f'(E)=L_0\otimes E$. Then
\[f'^*({\mc W}_{n,nd_0+d }(\mco))\cong \mcw(L_0).\]
The result now follows from \cite[Theorem 1]{li}.

Under the weaker assumption $nd_0+d\ge n(2g-1)$, consider the morphism
\[f'':\pic^0(C)\times \mcmx\longrightarrow\mcm:\ \ (L_1,E)\longmapsto L_1\otimes E.\]
This is a finite map, so $\mcw(L_0)$ is $\theta_{n,d}$-stable if $f''^*(\mcw(L_0))$ is $f''^*(\theta_{n,d})$-stable.

Now consider the restriction of $f''^*(\mcw(L_0))$ to a fibre $\pic^0(C)\times \{E_1\}$. From the definition, it follows that, if $L\in\pic^1(C)$,
\[f''^*(\mcw(L_0))|_{\pic^0(C)\times \{E_1\}}\cong\alpha_{L,E_1\otimes L}^*(\mcw(L_0)).\]
Since $L_0$ is a line bundle, it follows from Theorem \ref{c1}(b) that
$f''^*(\mcw(L_0))|_{\pic^0(C)\times \{E_1\}}$ is $\theta_{1,0}$-semistable. On the other hand, for $L_1\in\pic^0(C)$
\[f''^*(\mcw(L_0))|_{\{L_1\}\times\mcmx}\cong\mcwx(L_1\otimes L_0)\]
and this is $\theta_{n,\xi}$-stable by Theorem \ref{th2}(i). It follows from \cite[Proposition 4.8]{li} and \cite[Lemma 2.2]{bbn} that $\mcw(L_0)$ is $\theta_{n,d}$- stable.
\end{proof}

\begin{remark}\label{r2}\begin{em}\mbox{}
\begin{enumerate}
\item[(i)] Note that we require only one of the restrictions to be stable to apply \cite[Lemma 2.2]{bbn}; the other needs only to be semistable.

\item[(ii)] For $nd_0+n_0d\ge n_0n(2g-1)$, the same argument will prove that, if $\mcwx(E_0)$ is $\theta_{n,\xi}$-stable, then $\mcw(E_0)$ is $\theta_{n,d}$-stable.
\end{enumerate}
\end{em}\end{remark}

When $n_0\ge2$, the methods above do not currently work. Instead, we need to use an argument based on the use of spectral curves. Recall from \cite[Theorem 1 and Remarks 3.1 and 3.2]{bnr} (see also \cite[section 3.4]{li}) that, for any $n,d$, there exist a smooth irreducible $n$-sheeted covering $\pi:C'\longrightarrow C$ and an open set 
\[T^\delta:=\{L\in \pic^\delta(C')|\pi_*(L) \mbox{ is stable}\}\]such that the morphism $h:T^\delta
\longrightarrow \mcm$ defined by $h(L)=\pi_*(L)$ is dominant. Here
\begin{equation}\label{eq6}
\delta=d+n(n-1)(g-1),\ \ g(C')=n^2(g-1)+1
\end{equation}
and
\begin{equation}\label{eq7}
\pi_*({\mc O}_{C'})\cong {\mco}\oplus K_C^{-1}\oplus\cdots\oplus K_C^{1-n}.
\end{equation}

\begin{lemma}\label{l13}
Except when $g=n=2$ and $d$ is even, the complement of $T^\delta$ in $\pic^\delta(C')$ has codimension $\ge2$.
\end{lemma}

\begin{proof}
For $n\ge3$, it is proved in \cite[Remark 5.2]{bnr} that $\operatorname{codim}((T^\delta)^c) \ge2g-2$. For $n=2$, we can proceed as in this remark to obtain 
\[ \operatorname{codim}((T^\delta)^c)\ge g-1,\]
with strict inequality if $\delta$ is odd. This completes the proof.
\end{proof}

Let $\theta_{1,\delta}$ denote a $\theta$-bundle on $\pic^\delta(C')$. In \cite[Theorem 4.3]{li}, Li related the theta-bundle $\theta_{n,d}$ to $\theta_{1,\delta}$.

\begin{lemma}\label{l7}
Suppose that $\gcd(n,d)=1$ and let $\mc E$ be a vector bundle on $\mcm$. If $h^*({\mc E})$ extends to a $\theta_{1,\delta}$-stable (respectively, semistable) bundle on $\pic^\delta(C')$, then $\mc E$ is $\theta_{n,d}$-stable (respectively, semistable). Moreover, if $\phi _L^*(h^*(\mc E) )$ is stable for some $L\in\pic^\delta(C')$, then $\mc E$ is $\theta_{n,d}$-stable. 
\end{lemma}

\begin{proof} By \cite[Theorem 4.3]{li}, we have
\[h^*(\theta_{n,d})\cong\theta_{1,\delta}^n|_{T^\delta}.\]
Since $\dim T^\delta=n^2(g-1)+1=\dim\mcm$, the first part of the result now follows from \cite[Lemma 2.1]{bbn}. The second part follows from Lemma \ref{l2}.
\end{proof}

\begin{lemma}\label{l8}
Let $E$ be a stable (respectively, semistable) bundle on $C$. Then
$\pi^*(E)$ is stable (respectively, semistable) on $C'$.
\end{lemma}

\begin{proof}
For $E$ semistable, this is \cite[Theorem 2.4]{BS}.

Now assume that $E$ is stable. We know that $\pi^*(E)$ is polystable on $C'$
\cite[Proposition 2.3]{BS}. Moreover,
\[H^0(\operatorname{End}(\pi^*(E)))=H^0(\pi_*(\operatorname{End}(\pi^*(E))))=H^0(\operatorname{End}(E)\otimes \pi_*({\mc O}_{C'}))={\mathbb C},\]
where the last equality comes from \eqref{eq7} and the fact that $E$ is simple and $\operatorname{End}(E)$ is semistable of degree $0$.
So $\pi^*(E)$ is simple and therefore stable.
\end{proof}

\begin{lemma}\label{l9} Suppose that $\gcd(n,d)\,=\,1$, $nd_0+n_0d\,>\,n_0n(2g-2)$ and
$E_0\,\in\,{\mc M}_{n_0,d_0}$. Then
\[({\mc W}_{1,\delta}(\pi^*(E_0)))|_{T^\delta}\,\cong\, h^*(\mcw(E_0))\, .\]
\end{lemma}

\begin{proof}
Let ${\mcu}^\delta$ be a universal bundle on $\pic^\delta(C')\times C'$. Possibly after tensoring by a line bundle lifted from $T^\delta$, it follows from the definitions that
\[(1_{T^\delta}\times\pi)_*({\mcu}^\delta|_{T^\delta\times C'})\cong(h\times 1_C)^*(\mcu).\]
Now, tensoring both sides by $p_2^*(E_0)$ and taking direct images by $p_T$, we obtain
\[p_{T*}(p_2^*(\pi^*(E_0))\otimes{\mcu}^\delta|_{T^\delta\times C'})\cong p_{T*}(h\times 1_C)^*({\mcu}\otimes p_2^*(E_0))\]
on $T^\delta$. Using base change on the right hand side, this gives the result.
\end{proof}

\begin{proof}[Proof of Theorem \ref{th3}]
(i) is Lemma \ref{l6}.

(ii) Let $E_0\in{\mc M}_{n_0,d_0}$. By Lemma \ref{l8}, the bundle $\pi^*E_0$ on $C'$ is stable. Hence, by Theorem \ref{th1}(ii), ${\mc W}_{1,\delta}(\pi^*(E_0))$ is $\theta_{1,\delta}$-stable (respectively, semistable) on $\pic^\delta(C')$ provided that $nd_0+n_0\delta>n_0(2g(C')-1)$ (respectively, $\ge$). Using \eqref{eq6}, we see that this condition is equivalent to
\[nd_0+n_0d>n_0n(n+1)(g-1)+n_0 \mbox{ (respectively }\ge\mbox{)}.\]
The result now follows from Lemmas \ref{l7} and \ref{l9}.
\end{proof}

\begin{remark}\label{r6}\begin{em}
Suppose $nd_0+n_0\delta>n_0(g(C')-1)$. If ${\mc W}_{1,\delta}(\pi^*(E_0))$ were $\theta_{1,\delta}$-stable and Lemma \ref{l9} still held, then we would have $\mcw(E_0)$ $\theta_{n,d}$-stable for $nd_0+n_0d>n_0n(g-1)$. This is a plausible conjecture.
\end{em}\end{remark} 

\section{Deformations of Picard bundles: $n\ge2$}\label{defo}

In this section, we consider deformations of Picard bundles with a view to constructing and describing morphisms from moduli spaces of bundles on $C$ to moduli spaces of bundles on $\mcm$, thus obtaining fine moduli spaces for Picard bundles. We suppose throughout that 
\begin{equation}\label{eq73}n\ge2,\  \gcd(n,d)=1,\  nd_0+n_0d>n_0n(2g-2),
\end{equation}
 thus ensuring that our Picard sheaves are well defined and locally free. For technical reasons related to the use of Hecke correspondences in \cite{bisbn1}, we need to assume further that
\begin{equation}\label{eq72}
g\ge3;\ \ \mbox{if } g=3, n\ge4;\ \ \mbox{if } g=4, n\ge3.
\end{equation}
The following theorem summarises some relevant results from \cite{bisbn1}. Recall that the simple vector bundles over a scheme $X$ with fixed Chern character $\ch$ possess a coarse moduli space ${\mc M}_0(X)$ (see \cite[Corollary 6.5]{KO} for a proof of this fact in the analytic context - this space is possibly non-separated; by previous work of M. Artin, ${\mc M}_0(X)$ is an algebraic space). We define an equivalence relation on families of bundles $\{{\mc E}_s|s\in S\}$ parametrised by a scheme $S$ as follows: for two families ${\mc E}$, ${\mc E}'$, we write 
\begin{equation}\label{eq79}
{\mc E}\sim{\mc E}'\Longleftrightarrow{\mc E}\cong{\mc E}'\otimes p_S^*(L)\mbox{ for some line bundle }L\mbox{ on }S.
\end{equation}
As in \cite[Section 7]{bisbn1}, when $\gcd(n_0,d_0)=1$, let $\widehat{{\mc U}_\xi}$ be the bundle on $\mcmo\times\mcmx\times C$ defined by
\[\widehat{{\mc U}_\xi}:=p_{13}^*({\mc U}_0)\otimes p_{23}^*({\mc U}_\xi),\]
where ${\mc U}_0$ is a universal bundle on $\mcmo\times C$, and let 
\[\widehat{{\mc W}_\xi}:=p_{12*}(\widehat{{\mc U}_\xi}).\]

\begin{theorem}\label{th5} Suppose that \eqref{eq73} and \eqref{eq72} hold and let $\ch$ denote the Chern character of the Picard bundles $\mcwx(E_0)$. Then
\begin{itemize}
\item[(i)] for any vector bundle $E_0$ of rank $n_0$ and degree $d_0$ which is both semistable and simple, the Picard bundle $\mcwx(E_0)$ is simple;
\item[(ii)] the formula $E_0\mapsto\mcwx(E_0)$ defines a morphism
\[\beta_\xi:\mcmo\lra{\mc M}_0(\mcmx);\]
\item[(iii)] $\beta_\xi$ is an open immersion and, if $\gcd(n_0,d_0)=1$, it is an isomorphism onto a smooth connected component of ${\mc M}_0(\mcmx)$;
\item[(iv)] if $\gcd(n_0,d_0)=1$, $\mcmo$ is a fine moduli space for deformations of Picard bundles on $\mcmx$ with Chern character $\ch$ with respect to the equivalence relation $\sim$ defined in \eqref{eq79} with universal object $\widehat{{\mc W}_\xi}$.
\end{itemize}\end{theorem}

\begin{proof} (i) is \cite[Corollary 21]{bisbn1}.

(ii) follows from \cite[Theorem 24]{bisbn1} and (i). Note that the assumption that $\gcd(n_0,d_0)=1$ is not needed \cite[Remark 25]{bisbn1}. 

(iii) is not formally stated in \cite{bisbn1} except when $\mcwx(E_0)$ is $\theta_{n,\xi}$-stable for all $E_0\in\mcmo$ \cite[Theorem 26]{bisbn1}, but follows from \cite[Theorem 24]{bisbn1} and the fact that, if $\gcd(n_0,d_0)=1$, the moduli space $\mcmo$ is complete.

(iv) Let $\{{\mc W}_s\}$ be a family of Picard bundles of the form $\mcwx(E_0)$ parametrised by a scheme $S$. Then, by (iii), there exists a morphism $\phi:S\to\mcmo$ defined by $\phi(s)=\beta_\xi^{-1}({\mc W}_s)$. Now let ${\mc W}':=(\phi\times id_{\mcmx})^*(\widehat{{\mc W}_\xi})$. Then ${\mc W}'_s\cong{\mc W}_s$ for all $s\in S$. Since these bundles are simple, it follows that $L:=p_{S*}(\Hom({\mc W}',{\mc W}))$ is locally free of rank {1} (see \cite[Ch III Corollary 12.9]{ha}). It follows easily that the natural homomorphism $p_S^*(L)\otimes {\mc W}'\to{\mc W}$ is an isomorphism. So ${\mc W}'\sim{\mc W}$ and $\mcmo$ is a fine moduli space as required with universal object $\widehat{{\mc W}_\xi}$. 
\end{proof}

\begin{remark}\label{r71}\begin{em}
The restriction of $\widehat{{\mc W}_\xi}$ to $\{E_0\}\times\mcmx$ is $\mcwx(E_0)$, while its restriction to $\mcmo\times\{E\}$ is ${\mc W}_{n_0.d_0}(E)$. If $\mcwx(E_0)$ is $\theta_{n,\xi}$-stable for some $E_0$ and ${\mc W}_{n_0.d_0}(E)$ is $\theta_{n_0,d_0}$-semistable for some $E$, then $\widehat{{\mc W}_\xi}$ is $\theta$-stable for $\theta=a\theta_{n,\xi}+b\theta_{n_0,d_0}$ with $a,b>0$ by \cite[Lemma 2.2]{bbn}. This holds, by Theorem \ref{th2}(i) and Theorem \ref{th1}(ii), if $n_0=1$ and $nd_0+d\ge n(2g-1)$.
\end{em}\end{remark}

Our main object in this section is to obtain a similar result to Theorem \ref{th5} for $\mcm$. We begin with a lemma. Recall that $\pic^0(\mcm)$ denotes the group of topologically trivial line bundles on $\mcm$.

\begin{lemma}\label{l11}
Suppose that \eqref{eq73} and  \eqref{eq72} hold and let $E_0,E_0'\in\mcmo$, $L,L'\in\pic^0({\mcm})$. If $L\otimes\mcw(E_0)\cong L'\otimes\mcw(E_0')$, then $E_0\cong E_0'$ and $L\cong L'$. Moreover $\mcw(E_0)$ is simple for all $E_0\in\mcmo$.
\end{lemma}

\begin{proof} Suppose that $L\otimes\mcw(E_0)\cong L'\otimes\mcw(E_0')$ and consider the morphism $\det:\mcm\to\picd$. Since $\pic{\mcmx}\cong{\mathbb Z}$, 
\[L\cong {\det}^*(L_0), L'\cong{\det}^*(L_0')\mbox{ with }L_0, L_0'\in\pic^0(\picd).\]
So
\begin{eqnarray}\label{eq74}0&\ne& H^0(\mcm,L\otimes\mcw(E_0)\otimes (L'\otimes\mcw(E_0'))^*)\\\nonumber&=&H^0(\picd,{\det}_*(\mcw(E_0)\otimes\mcw(E_0')^*)\otimes L_0\otimes L_0'^*).
\end{eqnarray}
So ${\det}_*(\mcw(E_0)\otimes\mcw(E_0')^*)\ne0$ and hence
\[H^0(\mcmx,\mcwx(E_0)\otimes\mcwx(E_0')^*)\ne0\] 
for all $\xi\in\picd$.
It follows from \cite[Theorem 20]{bisbn1} that $H^0(E_0\otimes E_0'^*)\ne0$. Since $E_0$, $E_0'$ are stable of the same slope, this implies that $E_0\cong E_0'$. 

Note now that $\mcw(E_0)\otimes\mcw(E_0)^*$ contains a subbundle ${\mc O}_{\mcw}$ generated by the identity endomorphism of $\mcw(E_0)$. By \cite[Corollary 21]{bisbn1}, $\mcwx(E_0)$ is simple, so the inclusion of $H^0({\mc O}_{\mcw})$ in $H^0(\mcm,\mcw(E_0)\otimes\mcw(E_0)^*)$ restricts to an isomorphism ${\mathbb C}\to H^0(\mcmx,\mcwx(E_0)\otimes\mcwx(E_0)^*)$ for all $\xi\in\picd$. So 
\begin{equation}\label{eq76}
{\det}_*(\mcw(E_0)\otimes\mcw(E_0)^*)\cong{\mathcal O}_{\picd}
\end{equation}
 and, by \eqref{eq74},
\[H^0(\picd,L_0\otimes L_0'^*)\ne0.\]
Since $L_0$ and $L_0'$ are both topologically trivial, it follows that $L_0\cong L'_0$ and hence $L\cong L'$.

Finally, it follows from \eqref{eq76} that $H^0(\mcm,\End(\mcw(E_0)))\cong{\mathbb C}$, so $\mcw(E_0)$ is simple.
\end{proof}
 
\begin{remark}\label{r11}\begin{em}
Note that, for any $L\in\pic(\mcm)$, $p_{1*}(p_2^*(E_0)\otimes p_1^*(L)\otimes{\mc U})=L\otimes\mcw(E_0)$. Since ${\mc U}':=p_1^*(L)\otimes{\mc U}$ is a universal bundle on $\mcm\times C$, it follows that $L\otimes\mcw(E_0)$ is itself a Picard bundle.
\end{em}\end{remark}

\begin{remark}\label{r7}\begin{em} If we write $f$ for the fundamental class of the curve $C$, then the Chern character of $E_0$ is $n_0+d_0f$ and the Todd class of $C$ is $1-(g-1)f$, so the Chern character of $\mcw(E_0)$ is given by
\begin{equation}\label{eq75}
\ch:=\operatorname{ch}(\mcw(E_0))=\operatorname{ch}(\mcu)(n_0+d_0f)(1-(g-1)f)[C]
\end{equation}
by Grothendieck-Riemann-Roch. Moreover, if $L\in\pic^0(\mcm)$, then $L\otimes\mcw(E_0)$ also has Chern character $\ch$. Since the integral cohomology of $\mcm$ is torsion-free (see \cite{ab}), if $L\in\pic(\mcm)$ is not topologically trivial, the Chern character of $L\otimes\mcw(E_0)$ is different from $\ch$. Hence the Picard bundles with Chern character $\ch$ are precisely the bundles of the form $L\otimes\mcw(E_0)$ with $E_0\in\mcmo$ and $L\in\pic^0(\mcm)$. Note also that $n_0$, $d_0$ can be recovered from $\ch$.
\end{em}\end{remark}
 
 Now suppose that $\gcd(n_0,d_0)=1$. Consider the bundle $p_{13}^*({\mc U}_0)\otimes p_{23}^*({\mc U})$ on $\mcmo\times\mcm\times C$ and define 
\[\widehat{{\mc W}_{n,d,n_0,d_0}}:=p_{12*}(p_{13}^*({\mc U}_0)\otimes p_{23}^*({\mc U})).\]
This is a bundle on $\mcmo\times\mcm$ and can be regarded as a family of bundles on $\mcm$ parametrised by $\mcmo$; the members of this family are the Picard bundles $\mcw(E_0)$ for $E_0\in\mcmo$ (compare the bundle $\widehat{{\mc W}_\xi}$ on $\mcmo\times\mcmx$ used in the proof of Theorem \ref{th5}).  The definition is symmetrical; $\widehat{{\mc W}_{n,d,n_0,d_0}}$ can also be regarded as a family of bundles on $\mcmo$ parametrised by $\mcm$ with members $\mcwoo(E)$ for $E\in\mcm$. Now let ${\mc M}_0(\mcm)$ be the moduli space of simple bundles on $\mcm$ with Chern character $\ch$. Using $\widehat{{\mc W}_{n,d,n_0,d_0}}$ and Lemma \ref{l11}, we see that
\[\beta:\pic^0(\mcm)\times \mcmo\lra{\mc M}_0(\mcm):\ \ \beta(L,E_0)=L\otimes\mcw(E_0)\]
is a morphism.
Note that $\beta$ is defined even if $\gcd(n_0,d_0)\ne1$. To see this, recall that there exists an \'etale covering $X$ of $\mcmo$ such that a universal bundle exists on  $ X\times C$ \cite[Proposition 2.4]{NR0}. This implies the existence of a morphism from $\pic^0(\mcm)\times X$ to ${\mc M}_0(\mcm)$, which  descends to $\beta$.

\begin{lemma}\label{l20} Suppose that \eqref{eq73} and \eqref{eq72} hold. Then the differential of $\beta$ defines an isomorphism of Zariski tangent spaces
\begin{equation}\label{eq716}T(\pic^0(\mcm)\times \mcmo)_{(L,E_0)}\lra T{\mc M}_0(\mcm)_{\beta(L,E_0)}\end{equation}
for all $E_0\in \mcmo$, $L\in\pic^0(\mcm)$. In particular,
\begin{equation}\label{eq717}
\dim T{\mc M}_0(\mcm)_{\beta(L,E_0)}=g+n_0^2(g-1)+1.
\end{equation}
\end{lemma}

\begin{proof}
Recall that the tangent space to $\pic^0(\mcm)\times \mcmo$ at $(L,E_0)$ is 
\[H^1(\mcm,{\mc O}_{\mcm})\oplus H^1(\End(E_0)),\]
while that to ${\mc M}_0(\mcm)$ at $\beta(L,E_0)$ is $H^1(\mcm,\End(\mcw(E_0)))$.  
By \eqref{eq76}, we have $\det_*(\End(\mcw(E_0)))={\mc O}_{\picd}$. So the Leray spectral sequence of $\det$ yields an exact sequence
\begin{eqnarray}\label{eq71}0\lra H^1(\picd,{\mc O}_{\picd})\stackrel{\gamma}{\lra}& H^1(\mcm,\End(\mcw(E_0)))\stackrel{\delta}{\lra}\\\nonumber\stackrel{\delta}{\lra}& H^0(\picd, R^1_{\det }(\End(\mcw(E_0)))).
\end{eqnarray}
Since $H^1(\mcmx,{\mc O}_{\mcmx})=0$ for all $\xi$, we can identify $H^1(\picd,{\mc O}_{\picd})$ with $H^1(\mcm,{\mc O}_{\mcm})$. Under this identification, $\gamma$ is just the differential of $(\beta|_{{\pic^0(\mcm)}\times \{E_0\}})_L$. Moreover, for any $L\in\pic^0(\mcm)$, the differential of $(\beta|_{\{L\}\times \mcmo})_{E_0}$ gives a linear map
\[H^1(\End(E_0))\lra H^1(\mcm,\End(\mcw(E_0))),\]
which can be composed with $\delta$ to give a linear map
\begin{equation}\label{eq711}
H^1(\End(E_0))\lra H^0(\picd, R^1_{\det }(\End(\mcw(E_0)))).
\end{equation} 
This in turn yields a homomorphism of bundles
\begin{eqnarray}\label{eq712}
H^1(\End(E_0))\otimes{\mc O}_{\picd}\lra &H^0(\picd, R^1_{\det }(\End(\mcw(E_0))))\otimes{\mc O}_{\picd}\\\nonumber&\stackrel{ev}{\lra} R^1_{\det }(\End(\mcw(E_0))).
\end{eqnarray}
On the fibre of $\det$ over $\xi\in\picd$, \eqref{eq712} gives a linear map
\begin{equation}\label{eq713}
H^1(\End(E_0))\lra H^1(\mcmx,\End(\mcwx(E_0))).
\end{equation}
By construction, \eqref{eq713} is the infinitesimal deformation map for $\mcwx(E_0)$, which is an isomorphism for all $\xi$ by  \cite[Theorem 24]{bisbn1}, implying that \eqref{eq712} is an isomorphism. Hence \eqref{eq711} is also an isomorphism and the sequence \eqref{eq71} splits. It follows that \eqref{eq716} is an isomorphism. Since
\[h^1(\mcm,{\mc O}_{\mcm})=h^1({\mc M}_{1,d},{\mc O}_{{\mc M}_{1,d}})=h^1({\mc O}_C)=g\]
and $h^1(\End(E_0))=n_0^2(g-1)+1$, \eqref{eq717} now follows.
\end{proof}

\begin{remark}\label{r9}\begin{em}
Since the map $\delta$ in \eqref{eq71} is surjective, it follows from the Leray spectral sequence that there is an injective map
\[H^2(\picd,{\mc O}_{\picd})\lra H^2(\mcm,\End(\mcw(E_0))).\]
In particular $H^2(\mcm,\End(\mcw(E_0)))\ne0$, so in principle the infinitesimal deformations of $\mcw(E_0)$ could be obstructed. Lemma \ref{l20} shows that in fact they are not obstructed. Note also that, by \eqref{eq76},
\[R^0_{\det}(\End(\mcw(E_0)))={\det}_*(\mcw(E_0)\otimes\mcw(E_0)^*)\cong{\mathcal O}_{\picd}\]
and, by \eqref{eq712},
\[R^1_{\det }(\End(\mcw(E_0)))\cong H^1(\End(E_0))\otimes{\mc O}_{\picd}.\]
Moreover, it follows from \cite[Proposition 14 and (11)]{bisbn1} that $R^i_{\det}(\End(\mcw(E_0)))=0$ for all $E_0\in\mcmo$ if
\[2\le i\le \min\{r(n-r)(g-1)+(ne-r(d+1))\}-3,\]
the minimum being taken over all values of $r$, $e$ satisfying $0<r<n$, $rd\le ne$.  
\end{em}\end{remark}

\begin{corollary}\label{prop1} Suppose that \eqref{eq73} and \eqref{eq72} hold. Then $\beta$ is an open immersion and ${\mc M}_0(\mcm)$ is smooth of dimension $g+n_0^2(g-1)+1$ at $\beta(L,E_0)$.
\end{corollary}
\begin{proof}
Since $\pic^0(\mcmx)=\{{\mc O}_{\mcmx}\}$, we have a natural isomorphism
\begin{equation}\label{eq77}
\pic^0(\mcm)=\pic^0(\picd)\cong\pico.
\end{equation} Hence, by Lemma \ref{l11}, the dimension of $\Image\beta$ at $(L,E_0)$ is equal to $g+n_0^2(g-1)+1$. The result now follows from Lemma \ref{l20}.
\end{proof}

\begin{remark}\label{r10}\begin{em}
In \cite{bisbn1}, we obtained an inversion formula for $\beta_\xi$. We can use this to obtain a similar formula for $\beta$. In the first place, for any $L\in\pic^0(\mcm)$, we have $L\cong\det^*(L')$ with $L'\in\pic^0(\picd)$. It follows that $L\otimes\mcw(E_0)|_{\mcmx}\cong\mcwx(E_0)$ for any $\xi\in\picd$.  We can therefore use \cite[Theorem 19]{bisbn1} to recover $E_0$. Now, it follows from \eqref{eq76} that
\[{\det}_*(L\otimes\mcw(E_0)\otimes\mcw(E_0)^*)\cong L',\]
which recovers $L$. To make this formula precise, let $B:=\Image\beta$; this is an open subset of ${\mc M}_0(\mcm)$ and $\beta$ maps $\pic^0(\mcm)\times{\mc M}_{n_0,d_0}$ isomorphically to $B$ by Corollary \ref{prop1}. Now fix $\xi\in\picd$. Then, for any $F\in B$, we have 
\begin{eqnarray}\label{eq715}
\beta^{-1}(F)&=&({\det}^*({\det}_*(F\otimes\mcw(R^1_{p_2}(p_1^*(F|_{\mcmx})\otimes{\mc U}_\xi^*\otimes p_2^*(K_C)))^*)),\\
&&\nonumber\quad\quad\quad\quad\quad\quad\quad\quad\quad\quad\quad R^1_{p_2}(p_1^*(F|_{\mcmx})\otimes{\mc U}_\xi^*\otimes p_2^*(K_C))),
\end{eqnarray}
where $p_1$, $p_2$ are the projections of $\mcmx\times C$ onto its factors.
\end{em}\end{remark}

We can now prove the main theorem of this section, which is a refined version of Theorem \ref{th15}. Before stating the theorem, we make a definition. Suppose that \eqref{eq73} and \eqref{eq72} hold and $\gcd(n_0,d_0)=1$. If we let $\mc L$ denote a universal bundle on $\pic^0(\mcm)\times\mcm$, we can consider the bundle $p^*_{13}({\mc L})\otimes p^*_{24}({\mc U}_0)\otimes p^*_{34}(\mcu)$ on $\pic^0(\mcm)\times\mcmo\times\mcm\times\ C$ and define
\begin{equation}\label{eq714}
\widetilde{{\mc W}_{n,d,n_0,d_0}}:=p_{123*}(p^*_{13}({\mc L})\otimes p^*_{24}({\mc U}_0)\otimes p^*_{34}(\mcu)).
\end{equation}
This is a bundle on $\pic^0(\mcm)\times\mcmo\times\mcm\cong\pico\times\mcmo\times\mcm$.

\begin{theorem}\label{th4}
Suppose that \eqref{eq73} and \eqref{eq72} hold. 
\begin{itemize}
\item[(i)] Let ${\mc M}^0_0(\mcm)$ denote the irreducible component of ${\mc M}_0(\mcm)$ which contains $\Image\beta$. Then 
\begin{equation}\label{eq78}
\beta:\pic^0(\mcm)\times{\mc M}_{n_0,d_0}\lra{\mc M}^0_0(\mcm)
\end{equation}
is an injective birational morphism. 
\item[(ii)] If, in addition, $\gcd(n_0,d_0)=1$, then ${\mc M}^0_0(\mcm)$ is isomorphic to $\pic^0(\mcm)\times{\mc M}_{n_0,d_0}$ and is smooth of dimension $g+n_0^2(g-1)+1$. Moreover  $\pic^0(\mcm)\times{\mc M}_{n_0,d_0}$ is a fine moduli space for deformations of Picard bundles on $\mcm$ with respect to the equivalence relation $\sim$ defined in \eqref{eq79}. The bundle $\widetilde{{\mc W}_{n,d,n_0,d_0}}$ is a universal object for this moduli space.
\item[(iii)] If $\gcd(n_0,d_0)=1$ and $nd_0+n_0d>n_0n(n+1)(g-1)+n_0$, then ${\mc M}^0_0(\mcm)$ is an irreducible component of the moduli space of $\theta_{n,d}$-stable bundles on $\mcm$ with Chern character $\ch$. Moreover, if
 \begin{equation}\label{eq710}
 nd_0+n_0d>\max\{n_0n(n+1)(g-1)+n_0,n_0n(n_0+1)(g-1)+n\},
 \end{equation}
then  the universal bundle $\widetilde{{\mc W}_{n,d,n_0,d_0}}$ is $\theta$-stable for $\theta=a\theta_0+b\theta_{n,d}+c\theta_{n_0,d_0}$ with $a,b,c>0$, where $\theta_0$ is any ample divisor on $\pic^0(\mcm)$.
\end{itemize}
\end{theorem}
\begin{proof} (i) follows at once from Corollary \ref{prop1}. 

(ii) When $\gcd(n_0,d_0)=1$, $\pic^0(\mcm)\times{\mc M}_{n_0,d_0}$ is a smooth projective variety, so \eqref{eq78} is surjective and is therefore an isomorphism by Lemma \ref{l20}. We now argue exactly as in the proof of Theorem \ref{th5}(iv) using $\widetilde{{\mc W}_{n,d,n_0,d_0}}$ in place of $\widehat{{\mc W}_\xi}$, noting that the restriction of $\widetilde{{\mc W}_{n,d,n_0,d_0}}$ to any factor $\{L\}\times\{E_0\}\times\mcm$ is $L\otimes\mcw(E_0)$. 

(iii) The first statement follows from Theorem \ref{th3}(ii). Now suppose that \eqref{eq710} holds. The restriction of $\widetilde{{\mc W}_{n,d,n_0,d_0}}$ to $\pic^0(\mcm)\times\{E_0\}\times\{E\}$ is isomorphic to 
\[{\mc L}|_{\pic^0(\mcm)\times\{E\}}\otimes H^0(E_0\otimes E),\] 
which is $\theta_0$-semistable for any ample divisor $\theta_0$ on $\pic^0(\mcm)$. Moreover, the restriction to $\{L\}\times\mcmo\times\{E\}$ is just $\mcwoo(E)$ and the restriction to $\{L\}\times\{E_0\}\times\mcm$ is $L\otimes\mcw(E_0)$.  If $\mcw(E_0)$ is $\theta_{n,d}$-stable for some $E_0\in\mcmo$ and $\mcwoo(E)$ is $\theta_{n_0,d_0}$-stable for some $E\in\mcm$, then $\widetilde{{\mc W}_{n,d,n_0,d_0}}$ is $\theta$-stable for $\theta:=a\theta_0+b\theta_{n,d}+c\theta_{n_0,d_0}$ with $a,b,c>0$ by \cite[Lemma 2.2]{bbn}. This applies in particular if \eqref{eq710} holds
by Theorem \ref{th3}(ii).
\end{proof}

\begin{remark}\label{r72}\begin{em}
When $n_0=1$, we can use Theorem \ref{th3}(i) to replace the inequality in the first part of (iii) by $nd_0+d\ge n(2g-1)$. Moreover, using also Theorem \ref{th1}(ii), we can replace \eqref{eq710} by the same inequality $nd_0+d\ge n(2g-1)$.
\end{em}\end{remark}

\section{Deformation of Picard bundles: $n=1$}\label{defo1}
The proofs in section \ref{defo} depend heavily on results for Picard bundles on $\mcmx$ and do not work for $n=1$. In this case, we need to return to the work of Kempf \cite{k1} and Mukai \cite{mu} for the case $n_0=1$. The following theorem summarises their principal results with relevance to our problem.

\begin{theorem}\label{th6} Suppose that $g\ge2$ and $d_0+d>2g-2$. Then
\begin{itemize}
\item[(i)] for any line bundle $L_0$ of degree $d_0$, the Picard bundle $\mcwd(L_0)$ is simple;
\item[(ii)] the morphism
\[\beta:\pic^0({\mc M}_{1,d})\times{\mc M}_{1,d_0}\longrightarrow{\mc M}_0({\mc M}_{1,d}),\]
defined by $\beta(L,L_0)= L\otimes\mcwd(L_0)$, is injective;
\item[(iii)] the differential of $\beta$ at $(L,L_0)$ is injective;
\item[(iv)] if either $g=2$ or $g\ge3$ and $C$ is not hyperelliptic, then $\pic^0({\mc M}_{1,d})\times{\mc M}_{1,d_0}$ is a fine moduli space for the deformations of Picard bundles $\mcwd(L_0)$ with universal bundle $\widetilde{{\mc W}_{1,d,1,d_0}}$ as in \eqref{eq714}. Moreover, $\widetilde{{\mc W}_{1,d,1,d_0}}$ is $\theta$-stable for $\theta:=a\theta_0+b\theta_{1,d}+c\theta_{n_0,d_0}$ with $a,b,c>0$, where $\theta_0$ is any ample divisor on $\pic^0({\mc M}_{1,d})$.
\end{itemize}\end{theorem}
\begin{proof} (i) For $L_0={\mc O}_C$, this is \cite[Corollary 6.5]{k1}. In general, note that, if $T_0:{\mc M}_{1,d}\stackrel{\cong}{\to}{\mc M}_{1,d_0+d}$ is defined by $T_0(L)=L\otimes L_0$, then $T_0^*({\mc W}_{d_0+d}({\mc O}_C))={\mc W}_{1,d}(L_0)$. Alternatively, this follows directly from Theorem \ref{th1}(i). See also \cite[Corollary 2 to Theorem 13]{g1}.

(ii) This is easily deducible from \cite[Proposition 9.1]{k1}. Since we shall be presenting a proof for general $n_0$ later (Lemma \ref{l22}), we omit the details.
 
 (iii) is the first part of \cite[Theorem 8.4]{k1}.
 
 (iv) If $g\ge3$ and $C$ is not hyperelliptic, the second part of \cite[Theorem 8.4]{k1} states that the differential of $\beta$ at $(L,L_0)$ is an isomorphism for all $(L,L_0)$. The case $g=2$ is covered at the end of \cite{k1}. Thus, in these cases, $\beta$ is an isomorphism onto $\Image\beta$, which is an irreducible component of ${\mc M}_0({\mc M}_{1,d})$.  The identification of the universal bundle is proved as in Theorem \ref{th5}(iv) and Theorem \ref{th4}(ii). For the proof of stability of $\widetilde{{\mc W}_{1,d,1,d_0}}$, see Remark \ref{r71}.
\end{proof}

\begin{remark}\label{r12}\begin{em} Kempf defines Picard bundles also in negative degree; with our notation, this is equivalent to taking $d_0+d<0$ and defining the Picard bundle as $R^1_{p_1}(p_2^*(L_0)\otimes{\mc U})$. In \cite{mu}, Mukai adopts a similar approach and follows Schwarzenberger \cite{s} in choosing a point $c\in C$ and defining Picard sheaves $F_e$ over $\pico$ by
\[F_e:=R^1_{p_1}(p_2^*({\mc O}_C(ec))\otimes{\mc U}).\]
He then studies deformations of $F_e$ for $e\le g-1$. In our notation, this is equivalent to considering $R^1_{p_1}(p_2^*(L_0)\otimes{\mc U})$ for $d_0+d\le g-1$. For $d_0+d<0$, this sheaf is locally free; moreover, if $\phi:{\mc M}_{1,d}\to{\mc M}_{1,2g-2-d}$ is given by $\phi(L)=L^{-1}\otimes K_C$, then it is dual to $\phi^*({\mc W}_{1,2g-2-d}(L_0^{-1}))$ (for a suitable choice of universal bundle on ${\mc M}_{1,2g-2-d}$) by relative Serre duality. Note however that  $p_{1*}(p_2^*(L_0)\otimes{\mc U})=0$ for $d_0+d\le g-1$, so $L_0$ satisfies the weak index theorem (WIT) in this wider range, which is where Mukai works. Mukai also handles the hyperelliptic case, when $\beta$ still maps $\pic^0({\mc M}_{1,d})\times{\mc M}_{1,d_0}$ surjectively to a component of ${\mc M}_0({\mc M}_{1,d})$, but this component is non-reduced when $g\ge3$; in fact, its Zariski tangent spaces have dimension $3g-2$ \cite[Lemma 4.9 and Remark 4.17]{mu}. The proof of Theorem \ref{th6} depends on a classical theorem of Max Noether  \cite{noe} asserting that the multiplication map
\begin{equation}\label{eq93}\mu:H^0(K_C)\otimes H^0(K_C)\lra H^0(K_C^2)
\end{equation}
is surjective if $C$ is not hyperelliptic. Note that $\Coker(\mu)$ coincides with the Koszul cohomology group $K_{0,2}(C,K_C)$ (see, for example, \cite[p6]{an} for the definition). So the surjectivity of $\mu$ is equivalent to the vanishing of this Koszul cohomology group. This is in fact the first case of Green's Conjecture (see \cite[p.58]{an}). 
\end{em}\end{remark}

Our object now is  to generalise Theorem \ref{th6} to the case $n_0\ge2$. We assume from now on that $d_0+n_0d>n_0(2g-2)$. It follows that $\mcwd(E_0)$ is locally free for all $E_0\in\mcmo$. We shall not however assume that $n_0\ge2$ except where this is explicitly stated. Before proceeding, we state some facts about Picard varieties.

In the first place, $\pico$ is an abelian variety and ${\mc M}_{1,d}$ is a principal homogeneous space for $\pico$ under the action
\[{\mc M}_{1,d}\times\pico\lra{\mc M}_{1,d}\ :\ (L,M)\mapsto M\otimes L.\]
There exists a line bundle ${\mc P}$ on ${\mc M}_{1,d}\times\pico$ (the Poincar\'e bundle) which is universal for topologically trivial line bundles on ${\mc M}_{1,d}$ and also for topologically trivial line bundles on $\pico$. For $M\in\pico$, we write ${\mc P}_M$ for the bundle ${\mc P}|_{{\mc M}_{1,d}\times\{M\}}$. To fix ${\mc P}$, we suppose that ${\mc P}_{{\mc O}_C}\cong{\mc O}_{{\mc M}_{1,d}}$. We fix also a point $c\in C$. We have the following properties:
\begin{itemize}
\item[I] The map $\pico\to\pic^0({\mc M}_{1,d})$ defined by $M\mapsto{\mc P}_M$ is an isomorphism of algebraic groups.
\item[II] There is a closed immersion $\iota:C\to\pico$ defined by $\iota(x)={\mc O}_C(x-c)$ for which $\iota^*:\pic^0(\pico)\to \pic^0(C)$ is an isomorphism. We can therefore take
$(1_{{\mc M}_{1,d}}\times \iota)^*{\mc P}$ as our universal bundle ${\mc U}$ on ${\mc M}_{1,d}\times C$. It follows that, for all $x\in C$,
\[{\mc U}_x:={\mc U}|_{{\mc M}_{1,d}\times\{x\}}\cong{\mc P}_{\iota(x)}.\]
\item[III] For any $L\in\pic^0({\mc M}_{1,d})$, $L\not\cong{\mc O}_{{\mc M}_{1,d}}$,
\[H^i({\mc M}_{1,d},L)=0\]
for all $i$.
\end{itemize}
(The properties I and III may be found in a more general context in \cite[Section 8]{mf}, which also covers the existence of ${\mc P}$, and the definition of $\iota$ is clearly stated at the beginning of \cite[Section 4]{mu}.)

\begin{lemma}\label{l21}
Suppose that  $d_0+n_0d>n_0(2g-2)$ and let $E_0, E_0'\in\mcmo$. Then 
\begin{equation}\label{eq82}
H^0(E_0\otimes E_0'^*)\cong H^0(\picd,\mcwd(E_0)\otimes \mcwd(E_0')^*)
\end{equation}
and, for all $i\ge1$, there exists a short exact sequence
\begin{eqnarray}\label{eq83}
0\lra H^1(\wedge^{i-1}N\otimes E_0\otimes E_0'^*)\lra& H^i(\picd,\mcwd(E_0)\otimes\mcwd(E_0')^*)\\
\nonumber&\lra H^0(\wedge^iN\otimes E_0\otimes E_0'^*)\lra0,
\end{eqnarray}
where $N$ is the normal bundle to $\iota(C)$ in $\pico$ and is defined by the exact sequence
\begin{equation}\label{eq84}
0\lra N^*\lra H^0(K_C)\otimes {\mc O}_C\stackrel{\operatorname{ev}}{\lra} K_C\lra0.
\end{equation}
\end{lemma}
\begin{proof}
Following \cite{k1}, we define a sheaf ${\mc R}$ on $C\times\picd\times C$ by
\[{\mc R}:=p_{12}^*({\mc U})\otimes p_{23}^*({\mc U}^*)\otimes p_3^*(K_C).\]
For $x\ne y\in C$, we have ${\mc U}_x\not\cong{\mc U}_y$  by II and hence $H^i(\picd,{\mc U}_x\otimes{\mc U}_y^*)=0$  by III. It follows that $R^i_{p_{13}}({\mc R})$ is supported on the diagonal of $C\times C$ and therefore can have non-zero cohomology in dimensions $0$ and $1$ only. Applying the Leray spectral sequence for $p_{13}$ to ${\mc R}\otimes p_1^*(E_0)\otimes p_3^*(E_0')$, we therefore obtain, for all $i\ge1$, a short exact sequence
\begin{eqnarray}\label{eq85}
0&\lra& H^1(C\times C,R^i_{p_{13}}({\mc R})\otimes p_1^*(E_0)\otimes p_2^*(E_0'^*))\\
\nonumber&&\lra H^{i+1}(C\times\picd\times C,{\mc R}\otimes p_1^*(E_0)\otimes p_3^*(E_0'^*))\\
\nonumber&&\lra H^0(C\times C,R^{i+1}_{p_{13}}({\mc R})\otimes p_1^*(E_0)\otimes p_2^*(E_0'^*))\lra0.
\end{eqnarray}
Since  $R^i_{p_{13}}({\mc R})$ is supported on the diagonal of $C\times C$, we have
\[H^j(C\times C,R^i_{p_{13}}({\mc R})\otimes p_1^*(E_0)\otimes p_2^*(E_0'^*))\cong H^j(p_{1*}(R^i_{p_{13}}({\mc R}))\otimes E_0\otimes E_0'^*).\]
Now $p_{1*}(R^i_{p_{13}}({\mc R}))\cong \wedge^{i-1}N$ by \cite[Lemma 6.3]{k1} and
\[H^{i+1}(C\times\picd\times C,{\mc R}\otimes p_1^*(E_0)\otimes p_3^*(E_0'^*))\cong H^i(\picd,\mcwd(E_0)\otimes\mcwd(E_0')^*)\]
by \cite[Proposition 1 and Remark 2]{bisbn1}. Substituting in \eqref{eq85}, we obtain \eqref{eq82} and \eqref{eq83}.
Finally, \eqref{eq84} follows from the fact that the restriction of the tangent bundle of $\pico$ to $\iota(C)$ is naturally isomorphic to $H^1({\mc O}_C)\otimes{\mc O}_C\cong H^0(K_C)^*\otimes{\mc O}_C$.
\end{proof}
\begin{remark}\label{r82}\begin{em}
By \eqref{eq84}, the bundle $N$ is isomorphic to $M_{K_C}^*$ (see \eqref{eq1}).
\end{em}\end{remark}

\begin{corollary}\label{c2}
Under the hypotheses of Lemma \ref{l21}, 
\begin{itemize}
\item[(i)]$\mcwd(E_0)$ is simple for all $E_0\in\mcmo$.
 \item[(ii)]if $E_0\not\cong E_0'$, then $\mcwd(E_0)\not\cong\mcwd(E_0')$.
 \end{itemize}
\end{corollary}
\begin{proof}
(i) Take $E_0'=E_0$ in Lemma \ref{l21}. Since $E_0$ is stable, we have $H^0(E_0\otimes E_0^*)\cong{\mathbb C}$. The result now follows from \eqref{eq82}. (For $d_0+n_0d>n_0(2g-1)$, the result follows directly from Theorem \ref{th1}(ii), but this is not quite sufficient for us.)

(ii) Since $E_0$ and $E_0'$ are both stable of the same slope, $H^0(E_0\otimes E_0'^*)=0$. The result follows at once from \eqref{eq82}.
\end{proof}

In view of Corollary \ref{c2}(i), we have a morphism 
\[\beta:\pic^0(\picd)\times\mcmo\lra{\mc M}_0(\picd)\]
defined by $\beta(L,E_0)=L\otimes\mcwd(E_0)$.
We now extend Corollary \ref{c2}(ii) to prove the injectivity of $\beta$.
\begin{lemma}\label{l22}
Suppose that  $d_0+n_0d>n_0(2g-2)$ and let $E_0, E_0'\in\mcmo$, $L,L'\in\pic^0({\mc M}_{1,d})$. If $L\otimes\mcwd(E_0)\cong L'\otimes\mcwd(E_0')$, then $E_0\cong E_0'$ and $L\cong L'$. Hence $\beta$ is injective and
\begin{equation}\label{eq95}
\dim {\mc M}_0({\mc M}_{1,d})|_{\beta(L,E_0)}\ge g+n_0^2(g-1)+1.
\end{equation}
\end{lemma}

\begin{proof}
We prove first that, for any $L\in\pic^0({\mc M}_{1,d})$,
\begin{equation}\label{eq90}
H^g({\mc M}_{1,d},L\otimes\mcwd(E_0))\ne0 \Longleftrightarrow L\cong{\mc P}_M \mbox{ with }M=\iota(x)^*\mbox{ for some }x\in C.
\end{equation}

We consider the bundle ${\mc U}\otimes p_1^*(L)\otimes p_2^*(E_0)$ on ${\mc M}_{1,d}\times C$. By the universal property of ${\mc P}$, we can write $L={\mc P}_M$ for some $M\in\pico$. By Serre duality, we have
\[H^g({\mc M}_{1,d},{\mc O}_{{\mc M}_{1,d}})\cong H^0({\mc M}_{1,d},{\mc O}_{{\mc M}_{1,d}})\ne0.\]
From this and III, we see that, for any $x\in C$,
\[H^g({\mc M}_{1,d},{\mc P}_{\iota(x)}\otimes{\mc P}_M\otimes (E_0)_x)\ne0\Longleftrightarrow M\cong \iota(x)^*.\]
In particular, if $M\cong \iota(x)^*$, the sheaf $R^g_{p_2}({\mc U}\otimes p_1^*(L)\otimes p_2^*(E_0))$ is supported on the point $x$. Since also $R^{g+1}_{p_2}({\mc U}\otimes p_1^*(L)\otimes p_2^*(E_0))=0$, it follows from the base change theorem \cite[Section 5, Corollary 2]{mf} that 
\[H^g({\mc M}_{1,d}\times C,{\mc U}\otimes p_1^*(L)\otimes p_2^*(E_0))\ne0\Longleftrightarrow M\cong \iota(x)^*.\]
Noting that $p_{1*}({\mc U}\otimes p_1^*(L)\otimes p_2^*(E_0))=L\otimes\mcwd(E_0)$, while $R^1_{p_1}({\mc U}\otimes p_1^*(L)\otimes p_2^*(E_0))=0$ since $d_0+n_0d>n_0(2g-2)$, we deduce \eqref{eq90}.

To prove the lemma, we can clearly assume that $L'$ is trivial. In this case, $L'\cong{\mc P}_{{\mc O}_C}$, so, by II and \eqref{eq90}, we have $H^g({\mc M}_{1,d},\mcwd(E_0'))\ne0$. Using \eqref{eq90} again, we have $L\cong{\mc P}_M$ with $M=\iota(x)^*$ for some $x\in C$. Now let $L_1\in\pic^0({\mc M}_{1,d})$ and consider the isomorphism
\[L_1\otimes L\otimes\mcwd(E_0)\cong L_1\otimes\mcwd(E_0').\]
Writing $L_1={\mc P}_{M_1}$, the left-hand side has non-zero $H^g$ if and only if $M_1\otimes \iota(x)^*\cong \iota(y)^*$ for some $y\in C$, while the right-hand side has non-zero $H^g$ if and only if $M_1\cong \iota(z)^*$ for some $z\in C$. Hence, for any $z$, there exists $y\in C$ such that $\iota(x)^*\otimes \iota(z)^*\cong \iota(y)^*$, or, equivalently, $x+z\sim y+c$ as divisors on $C$. Given $x$, this must be true for any $z$, which implies $x=c$ since $g\ge2$. This proves that $L$ is trivial  and $\mcwd(E_0)\cong \mcwd(E_0')$. By Corollary \ref{c2}(ii), it follows that $\beta$ is injective and hence \eqref{eq95} holds.
\end{proof}

The next step is to estimate the dimension of the space of infinitesimal deformations of $\mcwd(E_0)$ at $\beta(L, E_0)$.

\begin{lemma}\label{l23} Suppose that $n_0\ge2$, $d_0+n_0d>n_0(2g-2)$ and let $E_0\in\mcmo$. Then
\begin{equation}\label{eq88}
\dim H^1(\picd,\mcwd(E_0)\otimes\mcwd(E_0)^*)\ge g+n_0^2(g-1)+1,
\end{equation}
with equality if and only if $C$ is non-hyperelliptic and $h^0(N\otimes\operatorname{ad}(E_0))=0$.
\end{lemma}

\begin{proof}
\eqref{eq88} follows from Lemma \ref{l22} and the fact that $H^1(\picd,\mcwd(E_0)\otimes\mcwd(E_0)^*)$ is the Zariski tangent space of ${\mc M}_0(\picd)$ at $\beta(L,E_0)$ for any $L\in\pic^0(\picd)$. Taking $i=1$ in \eqref{eq83}, we obtain a short exact sequence
\begin{eqnarray}\label{eq86}
0\lra H^1(E_0\otimes E_0^*)\stackrel{\gamma}{\lra}& H^1(\picd,\mcwd(E_0)\otimes\mcwd(E_0)^*)\\
\nonumber&\stackrel{\delta}{\lra} H^0(N\otimes E_0\otimes E_0^*)\lra0.
\end{eqnarray}
It follows that we have equality in \eqref{eq88} if and only if $h^0(N\otimes E_0\otimes E_0^*)=g$. Since $E_0\otimes E_0^*\cong{\mc O}_C\oplus\operatorname{ad}(E_0)$ and $h^0(N)\ge g$ by \eqref{eq84}, this holds if and only if $h^0(N)=g$ and $h^0(N\otimes\operatorname{ad}(E_0))=0$. Now $h^0(N)=g$ by Noether's Theorem (see Remark \ref{r12}) if $C$ is not hyperelliptic but this fails for $C$ hyperelliptic of genus $g\ge3$ (in fact, in this case $N\cong H^{\oplus(g-1)}$, where $H$ is the hyperelliptic line bundle, and $h^0(N)=2g-2>g$). On the other hand, if $g=2$, then $N\cong K_C$ and
\[h^0(N\otimes\operatorname{ad}(E_0))=h^1(\operatorname{ad}(E_0))=n_0^2-1>0\]
if $n_0\ge2$. This completes the proof.
\end{proof}

Before proceeding, we obtain an alternative condition for equality in \eqref{eq88} in terms of Koszul cohomology.
 
 \begin{proposition}\label{prop5} Suppose that $n_0\ge2$, $d_0+n_0d>n_0(2g-2)$ and let $E_0\in\mcmo$. Then
 \begin{equation}\label{eq91}
\dim H^1(\picd,\mcwd(E_0)\otimes\mcwd(E_0)^*)= g+n_0^2(g-1)+1+\dim(K_{0,2}(C;E_0\otimes E_0^*,K_C)).
\end{equation}
Hence, if $C$ is non-hyperelliptic, equality holds in (8.8) if and only if 
\[K_{0,2}(C;\operatorname{ad}(E_0),K_C)=0.\]
\end{proposition}
 
\begin{proof} Dualising the sequence \eqref{eq84} and tensoring by $E_0\otimes E_0^*$, we obtain an exact cohomology sequence
\begin{eqnarray}\label{eq87}
0\lra& H^0(K_C)^*\otimes H^0(E_0\otimes E_0^*)\lra H^0(N\otimes E_0\otimes E_0^*)\\
\nonumber&\lra H^1(K_C^*\otimes E_0\otimes E_0^*)\stackrel{\alpha}{\lra}H^0(K_C)^*\otimes H^1(E_0\otimes E_0^*).
\end{eqnarray}
Since $H^0(E_0\otimes E_0^*)\cong{\mathbb C}$, it follows that $h^0(N\otimes E_0\otimes E_0^*)=g$ if and only if $\alpha$ is injective. Now note that the dual of $\alpha$ is the multiplication map
\begin{equation}\label{eq89}
\mu_{E_0}:H^0(K_C)\otimes H^0(K_C\otimes E_0\otimes E_0^*)\lra H^0(K_C^2\otimes E_0\otimes E_0^*).
\end{equation}
So equality holds in \eqref{eq88} if and only if $\mu_{E_0}$ is surjective. In fact, since, by definition, 
\[\Coker(\mu_{E_0})=K_{0,2}(C;E_0\otimes E_0^*,K_C).\]
we have the general formula \eqref{eq91}.

Now suppose $C$ is non-hyperelliptic. Since $E_0\otimes E_0^*\cong {\mc O}_C\oplus\operatorname{ad}(E_0)$, it follows from Noether's Theorem (see Remark \ref{r12}) that 
\[\Coker(\mu_{E_0})=K_{0,2}(C;\operatorname{ad}(E_0),K_C).\]
Hence, equality holds in \eqref{eq88} if and only if this Koszul cohomology group vanishes.
\end{proof}

We are now ready to state the main results of this section. 

\begin{theorem}\label{th7}
Suppose that $g\ge2$, $n_0\ge2$ and $d_0+n_0d>n_0(2g-2)$.  Then
\begin{itemize}
\item[(i)] the morphism $\beta:\pic^0({\mc M}_{1,d})\times\mcmo\to {\mc M}_0({\mc M}_{1,d})$ is injective;
\item[(ii)] if $C$ is non-hyperelliptic and $h^0(N\otimes\operatorname{ad}(E_0))=0$, $\beta$ is an open immersion in the neighbourhood of $(L,E_0)$ for all $L\in\pic^0({\mc M}_{1,d})$;
\item[(iii)]if $C$ is non-hyperelliptic and $h^0(N\otimes\operatorname{ad}(E_0))=0$ for some $E_0\in\mcmo$, $\beta$ is an injective birational morphism from $\pic^0({\mc M}_{1,d})\times\mcmo$ to an irreducible component ${\mc M}_0^0({\mc M}_{1,d})$ of 
${\mc M}_0({\mc M}_{1,d})$;
\item[(iv)] if $C$ is non-hyperelliptic, $h^0(N\otimes\operatorname{ad}(E_0))=0$ for some $E_0\in\mcmo$ and $\gcd(n_0,d_0)=1$, $\beta$ is a bijective morphism onto ${\mc M}_0^0({\mc M}_{1,d})$. If $d_0+n_0d>n_0(2g-1)$, then ${\mc M}_0^0({\mc M}_{1,d})$ is a component of the moduli space of $\theta_{1,d}$-stable bundles on ${\mc M}_{1,d}$ with Chern character $\ch$.
\end{itemize}
\end{theorem}

\begin{proof} (i) This is Lemma \ref{l22}.

(ii) When $C$ is non-hyperelliptic and $h^0(N\otimes\operatorname{ad}(E_0))=0$, we have, by Lemmas \ref{l22} and \ref{l23}, that the dimension of ${\mc M}_0({\mc M}_{1,d})$ at $\beta(L,E_0)$ is equal to the dimension of its Zariski tangent space. Hence ${\mc M}_0({\mc M}_{1,d})$ is smooth at $\beta(L,E_0)$. The result now follows from Zariski's Main Theorem.

(iii) When $C$ is non-hyperelliptic and $h^0(N\otimes\operatorname{ad}(E_0))=0$ for some $E_0$, it follows from (ii) and the fact that $\pic^0({\mc M}_{1,d})\times\mcmo$ is irreducible that $\Image\beta$ is contained in a unique irreducible component ${\mc M}_0^0(\mcm)$ of ${\mc M}_0(\mcm)$ and maps birationally to this component. 

(iv) Since $\gcd(n_0,d_0)=1$,  $\pic^0({\mc M}_{1,d})\times\mcmo$ is complete. It follows from (iii) that $\beta$ maps bijectively to ${\mc M}_0^0(\mcm)$. The last part follows from Theorem \ref{th1}(ii).
\end{proof}

Before considering whether the hypothesis of Theorem \ref{th7}(iii) is satisfied, we will show that, at least in many cases, there exists $E_0\in\mcmo$ such that $H^0(N\otimes\operatorname{ad}(E_0))>0$.

\begin{proposition}\label{prop2}
Suppose that $g\ge2$, $n_0\ge2$, $d_0+n_0d>n_0(2g-2)$ and $d_0\equiv1\bmod n_0$ or $-1\bmod n_0$. Then there exists $E_0\in \mcmo$ such that $h^0(N\otimes\operatorname{ad}(E_0))>0$. Hence, at least in these cases, $\beta$ does not map $\pic^0({\mc M}_{1,d})\times\mcmo$ isomorphically to ${\mc M}_0^0(\mcm)$.
\end{proposition}
\begin{proof} Suppose first that $d_0\equiv1\bmod n_0$ and write $d_0=en_0+1$. Let $M_0,\cdots,M_{n_0-2}$ be mutually non-isomorphic line bundles of degree $e$ on $C$. We consider extensions
\begin{equation}\label{eq07}
0\lra M_0\oplus\cdots\oplus M_{n_0-2}\lra E_0\lra M_0(p)\lra0
\end{equation}
with $p\in C$.
Note that ${\mc O}_C(-p)=\Hom(M_0(p),M_0)\subset \operatorname{ad}(E_0)$. Hence $N(-p)\subset N\otimes\operatorname{ad}(E_0)$. Since $N$ has rank $g-1$ and $h^0(N)\ge g$, we have $h^0(N(-p))>0$, so $h^0(N\otimes\operatorname{ad}(E_0))>0$. It remains to prove that $E_0$ can be stable. 

If $F$ is a subbundle of $E_0$ contradicting stability, $F$ must map surjectively to $M_0(p)$ and the kernel of the homomorphism $F\to M_0(p)$ must contradict stability. We therefore have an exact sequence
\begin{equation}\label{eq96}
0\lra M_{i_1}\oplus\cdots\oplus M_{i_s}\lra F\lra M_0(p)\lra0
\end{equation}
with $s<n_0-1$. The extensions \eqref{eq07} are classified by $(n_0-1)$-tuples $(e_0,\cdots,e_{n_0-2})$ with $e_i\in H^1(M_i\otimes M_0^*(-p))$. Since the $M_i$ are mutually non-isomorphic, the existence of \eqref{eq96} implies that $e_j=0$ for some $j$. This is not true for the general extension.

For $d_0\equiv-1\bmod n_0$, one can replace $E_0$ by $E_0^*$ and use the argument above.

The last statement follows from Lemma \ref{l23}.
\end{proof}

\begin{remark}\label{r83}\begin{em}(i)
Propositions \ref{prop5} and \ref{prop2} show that there exist stable bundles $E_0$ such that $K_{0,2}(C;E_0\otimes E_0^*,K_C)\ne0$. This contrasts with the situation for line bundles, where $K_{0,2}(C,K_C)=0$ by Noether's Theorem.

(ii) More precisely, if $E_0$ is as in Proposition \ref{prop2}, then, for any $L\in\pic^0({\mc M}_{1,d})$, there exist infinitesimal deformations of $L\otimes\mcwd(E_0)$ which do not arise from deformations of the pair $(L,E_0)$.
\end{em}\end{remark} 

As a result of this proposition, we cannot expect to strengthen the results of Theorem \ref{th7}(iii) and (iv). However, Montserrat Teixidor i Bigas  \cite{t} has proved that $\mu_{E_0}$ is surjective when both $C$ and $E_0$ are general. Her proof involves degenerating to a chain of elliptic curves. We therefore have the following corollary to Theorem \ref{th7}.

\begin{corollary}\label{th8}
Let $C$ be a general curve of genus $g\ge3$ and suppose that $n_0\ge2$ and $d_0+n_0d>n_0(2g-2)$. Then $\beta$ is an injective birational morphism from $\pic^0({\mc M}_{1,d})\times\mcmo$ to an irreducible component ${\mc M}_0^0({\mc M}_{1,d})$ of 
${\mc M}_0({\mc M}_{1,d})$.  If $\gcd(n_0,d_0)=1$, $\beta$ is a bijective morphism onto ${\mc M}_0^0({\mc M}_{1,d})$.
\end{corollary}
\begin{proof}
This follows from Theorem \ref{th7}(i), (iii) and (iv) and \cite{t}.
\end{proof}


\end{document}